\documentclass[final]{article}

\usepackage{my_macros}

\newcommand{\emone}{\bf}
\newcommand{\emtwo}{\em}

\usepackage[final]{neurips_2020_edited}

\usepackage[utf8]{inputenc} 
\usepackage[T1]{fontenc}    
\usepackage[pagebackref]{hyperref}       
\usepackage{url}            
\usepackage{booktabs}       
\usepackage{amsfonts}       
\usepackage{nicefrac}       
\usepackage{microtype}      
\usepackage{enumitem}
\usepackage{cleveref}
\usepackage{multirow}
\usepackage{threeparttable,adjustbox}

\usepackage{hyperref}
\hypersetup{
	colorlinks=true,
	linkcolor=darkblue,
	citecolor=darkblue,
	urlcolor=darkgreen
    }

\usepackage{algorithm}
\usepackage[noend]{algpseudocode}

\hypersetup{
  colorlinks=true,
  linkcolor=blue,
  filecolor=magenta,
  urlcolor=cyan,
  citecolor=blue
}
\newcommand{\prm}[1]{\pi_{#1}}
\newcommand\et[1]{{\mathbb{E}_{t} \left[#1\right]}}
\def\<#1,#2>{\langle #1,#2\rangle}

\newcommand{\sigmaesc}{\sigma_{\ast}}
\newcommand{\sigmass}{\sigma_{\mathrm{Shuffle}}^2}

\usepackage{array}

\newcolumntype{C}[1]{>{\centering\let\newline\\\arraybackslash\hspace{0pt}}m{#1}}

\usepackage{multirow}

\setcitestyle{authoryear,round,citesep={;},aysep={,},yysep={;}}

\definecolor{mygreen}{rgb}{0.0, 0.65, 0.31}
\definecolor{mydarkgreen}{RGB}{39,130,67}
\usepackage[colorinlistoftodos,bordercolor=orange,backgroundcolor=orange!20,linecolor=orange,textsize=scriptsize]{todonotes}

\title{Random Reshuffling: Simple Analysis \\ with Vast Improvements}

\author{%
  Konstantin Mishchenko \\
  KAUST \\
  Thuwal, Saudi Arabia
  \And
  Ahmed Khaled \\
  Cairo University \\
  Giza, Egypt
  \And
  Peter Richt\'{a}rik \\
  KAUST \\
  Thuwal, Saudi Arabia
}

\begin{document}

\maketitle

\begin{abstract}
  Random Reshuffling (RR) is an algorithm for minimizing finite-sum functions that utilizes iterative gradient descent steps in conjunction with data reshuffling. Often contrasted with its sibling Stochastic Gradient Descent (SGD), RR is usually faster in practice and enjoys significant popularity in convex and non-convex optimization. The convergence rate of RR has attracted substantial attention recently and, for strongly convex and smooth functions, it was shown to converge faster than SGD if 1) the stepsize is small, 2) the gradients are bounded, and 3) the number of epochs is large. We remove these 3 assumptions, improve the dependence on the condition number from $\kappa^2$ to $\kappa$ (resp.\ from $\kappa$ to $\sqrt{\kappa}$) and, in addition, show that RR has a different type of variance. We argue through theory and experiments that the new variance type gives an additional justification of the superior performance of RR. To go beyond strong convexity, we present several results for non-strongly convex and non-convex objectives. We show that in all cases, our theory improves upon existing literature. Finally, we prove fast convergence of the Shuffle-Once (SO) algorithm, which shuffles the data only once, at the beginning of the optimization process. Our theory for strongly convex objectives tightly matches the known lower bounds for both RR and SO and substantiates the common practical heuristic of shuffling once or only a few times. As a byproduct of our analysis, we also get new results for the Incremental Gradient algorithm (IG), which does not shuffle the data at all.
\end{abstract}

\section{Introduction}

We study the finite-sum minimization problem
\begin{equation}
  \label{eq:finite-sum-min}
\textstyle  \min_{x \in \R^d} \Bigl [ f(x) = \frac{1}{n} \sum \limits_{i=1}^{n} f_{i} (x) \Bigr ],
\end{equation}
where each $f_{i}: \R^d \to \R$ is differentiable and smooth, and  are particularly interested in the big data machine learning setting where  the number of functions $n$ is large. Thanks to their scalability and low memory requirements, first-order methods are especially popular in this setting \citep{Bottou2018}. \emph{Stochastic} first-order algorithms in particular have attracted a lot of attention in the machine learning community and are often used in combination with various practical heuristics. Explaining these heuristics may lead to further development of stable and efficient training algorithms. In this work, we aim at better and sharper theoretical explanation of one intriguingly simple but notoriously elusive heuristic: {\em data permutation/shuffling}.

\subsection{Data permutation}

In particular, the goal of our paper is to obtain deeper theoretical understanding of  methods for solving  \eqref{eq:finite-sum-min} which rely on random or deterministic {\em permutation/shuffling} of the data  $\{1,2,\dots,n\}$ and perform incremental  gradient updates following the permuted order.  We study three methods which belong to this class, described next.

An immensely popular but theoretically elusive  method  belonging to the class of data permutation methods  is the {\bf Random Reshuffling (RR)} algorithm (see Algorithm~\ref{alg:rr}). This is the method we pay most attention to in this work, as reflected in the title. In each epoch $t$ of RR, we sample indices $\pi_{0}, \pi_{1}, \ldots, \pi_{n-1}$ \emph{without replacement} from $\{ 1, 2, \ldots, n \}$, i.e., $\{\pi_{0}, \pi_{1}, \ldots, \pi_{n-1}\}$ is a random permutation of the set $\{ 1, 2, \ldots, n \}$, and proceed with $n$ iterates of the form
\[ x_{t}^{i+1} = x_t^i - \gamma \nabla f_{\pi_{i}} (x_t^i), \]
where $\gamma > 0$ is a stepsize. We then set $x_{t+1} = x_t^n$,  and repeat the process for a total of $T$ epochs. Notice that in  RR, a {\em new} permutation/shuffling is generated at the beginning of each epoch, which is why the term  {\em re}shuffling is used.

Furthermore, we consider the {\bf Shuffle-Once (SO)} algorithm, which is identical to RR with the exception that it shuffles the dataset only once---at the very beginning---and then reuses this random permutation in all subsequent epochs (see Algorithm~\ref{alg:so}). Our results for SO follow as corollaries of the tools we developed in order to conduct a sharp analysis of RR.

Finally, we also consider the {\bf Incremental Gradient (IG)} algorithm, which is identical to SO, with the exception that the initial permutation is not random but deterministic. Hence, IG performs incremental gradient steps through the data in a {\em cycling} fashion. The ordering could be {\em arbitrary}, e.g., it could be  selected {\em implicitly} by the ordering the data comes in, or chosen {\em adversarially}. Again, our results for IG follow as a byproduct of our efforts to understand RR.

\begin{minipage}[t]{0.49\textwidth}
  \begin{algorithm}[H]
    \caption{Random Reshuffling (RR)}
    \label{alg:rr}
 \begin{algorithmic}[1]
   \Require Stepsize $\gamma > 0$, initial vector $x_0 = x_0^0 \in \R^d$, number of epochs $T$
    \For{epochs $t=0,1,\dotsc,T-1$}
   \State {\color{mydarkred}Sample a permutation $\prm{0}, \prm{1}, \ldots, \prm{n-1}$ of $\{ 1, 2, \ldots, n \}$    }
       \For{$i=0, 1, \ldots, n-1$}
          \State $x_{t}^{i+1} = x_t^{i} - \gamma \nabla f_{\prm{i}} (x_t^i)$
       \EndFor
       \State $x_{t+1} = x_{t}^{n}$
    \EndFor
 \end{algorithmic}
 \end{algorithm}
\end{minipage}
  \hfill
\begin{minipage}[t]{0.49\textwidth}
  \begin{algorithm}[H]
    \caption{Shuffle Once (SO)}
    \label{alg:so}
 \begin{algorithmic}[1]
   \Require Stepsize $\gamma > 0$, initial vector $x_0 = x_0^0 \in \R^d$, number of epochs $T$
   \State {\color{mydarkred} Sample a permutation $\prm{0}, \prm{1}, \ldots, \prm{n-1}$ of $\{ 1, 2, \ldots, n \}$}
    \For{epochs $t=0,1,\dotsc,T-1$}
       \For{$i=0, 1, \ldots, n-1$}
          \State $x_{t}^{i+1} = x_t^{i} - \gamma \nabla f_{\prm{i}} (x_t^i)$
       \EndFor
       \State $x_{t+1} = x_{t}^{n}$
    \EndFor
 \end{algorithmic}
 \end{algorithm}
\end{minipage}

\subsection{Brief literature review}

RR is usually contrasted with its better-studied sibling Stochastic Gradient Descent (SGD), in which each $\pi_{i}$ is sampled uniformly \emph{with replacement} from $\{ 1, 2, \ldots, n \}$. RR often converges faster than SGD on many practical problems~\citep{Bottou2009,Recht2013}, is more friendly to cache locality~\citep{Bengio2012}, and is in fact standard in deep learning~\citep{Sun2020}.

The convergence properties of SGD are well-understood, with tightly matching lower and upper bounds in many settings~\citep{Rakhlin2012,Drori2019,HaNguyen2019}. Sampling without replacement allows RR to leverage the finite-sum structure of \eqref{eq:finite-sum-min} by ensuring that \emph{each} function contributes to the solution once per epoch. On the other hand, it also introduces a significant complication: the steps are now {\em biased}. Indeed, in any iteration $i>0$ within an epoch, we face the challenge of not having (conditionally) unbiased gradients since
\[ \ec{\nabla f_{\pi_{i}} (x_t^i) \mid x_t^i} \neq \nabla f(x_t^i). \]
This bias implies that individual iterations do not necessarily approximate a full gradient descent step. Hence, in order to obtain meaningful convergence rates for RR, it is necessary to resort to more involved proof techniques. In recent work, various convergence rates have been established for RR. However, a satisfactory, let alone complete, understanding of the algorithm's convergence remains elusive. For instance, the early line of attack pioneered by  \citet{RR-conjecture2012} seems to have hit the wall as their noncommutative arithmetic-geometric mean conjecture is not true~\citep{Recht-Re_conj_is_false_ICML_2020}. The situation is even more pronounced with the SO method, as \citet{Safran2020good} point out that there are no convergence results specific for the method, and the only convergence rates for SO follow by applying the worst-case bounds of IG. \citet{Rajput2020} state that a common practical heuristic is to use methods like SO that do not reshuffle the data every epoch. Indeed, they add that 
\emph{``current theoretical bounds are insufficient to explain this phenomenon, and a new theoretical breakthrough may be required to tackle it''.}

IG has a long history owing to its success in training neural networks~\citep{Luo1991, Grippo1994}, and its asymptotic convergence has been established early~\citep{Mangasarian1994,Bertsekas2000}. 
Several rates for non-smooth \& smooth cases were established by \citet{Nedic2001, Li2019, Gurbuzbalaban2019IG, Ying2019} and \citet{Nguyen2020}.
Using IG poses the challenge of choosing a specific permutation for cycling through the iterates, which \citet{Nedic2001} note to be difficult. \citet{Bertsekas2011} gives an example that highlights the susceptibility of IG to bad orderings compared to RR. Yet, thanks to \citet{Gurbuzbalaban2019RR} and \citet{haochen2018random}, RR is known to improve upon both SGD and IG for \emph{twice}-smooth objectives. \citet{Nagaraj2019} also study convergence of RR for smooth objectives, and \citet{Safran2020good,Rajput2020} give lower bounds for RR and related methods.

\section{Contributions} \label{sec:contributions}

In this work, we study the convergence behavior of the data-permutation methods RR, SO and IG. While existing proof techniques succeed in obtaining insightful bounds for RR and IG, they fail to fully capitalize on the intrinsic power reshuffling and shuffling offers, and are not applicable to SO at all\footnote{As we have mentioned before, the best known bounds for SO are those which apply to IG also, which means that the randomness inherent in SO is wholly ignored.}. Our proof techniques are dramatically novel, simple, more insightful, and lead to improved convergence results, all under weaker  assumptions on the objectives than prior work.

\subsection{New and improved convergence rates for RR, SO and IG} In Section~\ref{sec:conv-theory}, we analyze the RR and SO methods and present novel convergence rates for strongly convex, convex, and non-convex  smooth objectives. Our results for RR are summarized in Table~\ref{tab:conv-rates}. 
  
  \begin{itemize}[leftmargin=0.15in,itemsep=0.01in,topsep=0pt]
    \item {\emone Strongly convex case.} If each $f_i$ is strongly convex, we introduce a {\emtwo new proof technique} for studying the convergence of RR/SO   that allows us to obtain a {\emtwo better dependence on problem constants}, such as the number of functions $n$ and the condition number $\kappa$, compared to prior work (see Table~\ref{tab:conv-rates}). Key to our results is a {\emtwo new notion of variance specific to RR/SO} (see Definition~\ref{def:bregman-div-noise}), which we argue explains the superior convergence of RR/SO compared to SGD in many practical scenarios. Our result for SO tightly {\emtwo matches the lower bound} of \citet{Safran2020good}. We prove similar results in the more general setting when each $f_{i}$ is convex and $f$ is strongly convex (see Theorem~\ref{thm:only-f-sc-rr-conv}), but in this case we are forced to use smaller stepsizes. 

    \item {\emone Convex case.}  For convex but not necessarily strongly convex objectives $f_{i}$, we give the first result showing that {\emtwo RR/SO can provably achieve better convergence than SGD} for a large enough number of iterations. This holds even when comparing against results that assume second-order smoothness, like the result of \citet{haochen2018random}.
    
    \item {\emone Non-convex case.}  For non-convex objectives $f_{i}$, we obtain for RR a  {\emtwo much better dependence on the number of functions $n$} compared to the prior work of \citet{Nguyen2020}. 
  \end{itemize}

Furthermore, in the appendix we formulate and prove convergence results for IG for strongly convex objectives, convex, and non-convex objectives as well. The bounds are worse than RR by a factor of $n$ in the noise/variance term, as IG does not benefit from randomization. Our result for strongly convex objectives {\em tightly matches the lower bound of \citet{Safran2020good} up to an extra iteration and logarithmic factors, and is the first result to tightly match this lower bound.}

\subsection{More general assumptions on the function class} Previous non-asymptotic convergence analyses of RR either obtain worse bounds that apply to IG, e.g., \citep{Ying2019,Nguyen2020}, or depend crucially on the assumption that each $f_i$ is Lipschitz \citep{Nagaraj2019,haochen2018random,Ahn2020}. Unfortunately, requiring each $f_i$ to be Lipschitz contradicts strong convexity~\citep{nguyen2018sgd} and is furthermore not satisfied in least square regression, matrix factorization, or for neural networks with smooth activations. In contrast, our work is the first to show how to leverage randomization to obtain better rates for RR without assuming each $f_i$ to be Lipschitz. In concurrent work, \citet{AhnYun2020} also obtain a result for non-convex objectives satisfying the Polyak-{\L}ojasiewicz inequality, a generalization of strong convexity. Their result holds without assuming bounded gradients or bounded variance, but unfortunately with a worse dependence on $\kappa$ and $n$ when specialized to $\mu$-strongly convex functions.
  \begin{itemize}[leftmargin=0.15in,itemsep=0.01in,topsep=0pt]
    \item {\emone Strongly convex and convex case.} For strongly convex and convex objectives {\emtwo we do not require any assumptions on the functions used beyond smoothness and convexity.} 
    \item {\emone Non-convex case.} For non-convex objectives we obtain our results under a significantly more general assumption than the bounded gradients assumptions employed in prior work. Our assumption is also provably satisfied when each function $f_i$ is lower bounded, and hence is {\emtwo not only more general but also a more realistic assumption to use.}
\end{itemize}

\begin{table}[]
  \begin{threeparttable}[b]
    \centering
    \caption{Number of individual gradient evaluations needed by RR to reach an $\e$-accurate solution (defined in Section~\ref{sec:conv-theory}). Logarithmic factors and constants that are not related to the assumptions are ignored. For non-convex objectives, $A$ and $B$ are the constants given by Assumption~\ref{asm:2nd-moment}.}
    \label{tab:conv-rates}
    \begin{tabular}{@{}C{1cm}@{}C{1cm}@{}cccc@{}}
    \toprule
      \multicolumn{2}{c}{Assumptions} & $\mu$-Strongly & Non-Strongly & \multirow[c]{2}{*}{Non-Convex} & \multirow[c]{2}{*}{Citation}\\ \cmidrule(lr){1-2}
      N.L.\tnote{\color{red}(1)} & U.V.\tnote{\color{red}(2)} & Convex & Convex & & \\ \midrule
        \cmark & \cmark & {\footnotesize $\kappa^2 n  + \frac{\kappa n  \sigmaesc}{\mu\sqrt{\e}}$} & --  & -- & {\footnotesize \citet{Ying2019}} \\[5pt]
        \xmark & \xmark & {\footnotesize $\kappa^2 n + \frac{\kappa \sqrt{n}  G}{\mu \sqrt{\e}}$} & {\footnotesize $\frac{L D^2 }{\e} + \frac{G^2 D^2}{\e^2}$}\tnote{\color{red}(3)} & -- & {\footnotesize \citet{Nagaraj2019}} \\[5pt]
        \xmark & \xmark & -- & -- & {\footnotesize $\frac{ L n}{\e^2} + \frac{ L n G}{\e^3}$} & {\footnotesize \citet{Nguyen2020}} \\[5pt]
        \cmark & \cmark & {\footnotesize $\frac{ \kappa^2 n }{\sqrt{\mu \e}}  + \frac{ \kappa^2 n \sigmaesc}{\mu \sqrt{\e}}$} \tnote{\color{red}(4)} & -- & -- & {\footnotesize \citet{Nguyen2020}} \\[5pt]
        \xmark & \xmark & {\footnotesize$\frac{\kappa \alpha }{\e^{1/\alpha}} + \frac{\kappa  \sqrt{n} G \alpha^{3/2}}{\mu\sqrt{\e}}$}\tnote{\color{red}(5)} & -- & -- & {\footnotesize \citet{Ahn2020}} \\[5pt]
        \cmark & \cmark & { \footnotesize$ \kappa n + \frac{\sqrt{n}}{\sqrt{\mu\e}} + \frac{\kappa \sqrt{n} G_0}{\mu \sqrt{\e}}$ }\tnote{\color{red}(6)} & -- & -- & {\footnotesize \citet{AhnYun2020}} \\[5pt]
        \cmidrule(lr){1-5}
        \multirow[c]{2}{*}[-4pt]{\cmark} & \multirow[c]{2}{*}[-4pt]{\cmark} & {\footnotesize $\kappa + \frac{\sqrt{\kappa n} \sigmaesc}{\mu \sqrt{\e}}$}\tnote{\color{red}(7)} &  \multirow[c]{2}{*}{$\frac{ L n}{\e} + \frac{ \sqrt{L n} \sigmaesc}{\e^{3/2}}$} & { \footnotesize \multirow[c]{2}{*}{$\frac{ L n}{\e^2} + \frac{ L \sqrt{n} (B+\sqrt{A})}{\e^3}$}} & \multirow[c]{2}{*}[-4pt]{This work} \\[5pt]
        & & { \footnotesize $\kappa n + \frac{\sqrt{\kappa n} \sigmaesc}{\mu \sqrt{\e}}$ } &  &  & \\
        \bottomrule
    \end{tabular}
    \begin{tablenotes}
      {\footnotesize
        \item [{\color{red}(1)}] Support for non-Lipschitz functions (N.L.): proofs without assuming that $\max_{i=1,\dotsc,n}\norm{\nabla f_{i} (x)} \leq G$ for all $x \in \R^d$ and some $G>0$. Note that $ \frac{1}{n} \sum_{i=1}^{n} \sqn{\nabla f_{i} (x_\ast)} \eqdef \sigmaesc^2 \leq G^2$ and $B^2 \leq G^2$.
        \item [{\color{red}(2)}] Unbounded variance (U.V.): there may be no constant $\sigma$ such that Assumption~\ref{asm:2nd-moment} holds with $A = 0$ and $B = \sigma$. Note that when the individual gradients are bounded, the variance is automatically bounded too.
        \item [{\color{red}(3)}] \citet{Nagaraj2019} require, for non-strongly convex functions, projecting at each iteration onto a bounded convex set of diameter $D$. We study the unconstrained problem.
        \item [{\color{red}(4)}] For strongly convex, \citet{Nguyen2020} bound $f(x)-f(x_\ast)$ rather than squared distances, hence we use strong convexity to translate their bound into a bound on $\|x-x_\ast\|^2$.
        \item [{\color{red}(5)}] The constant $\alpha > 2$ is a parameter to be specified in the stepsize used by \citep{Ahn2020}. Their full bound has several extra terms but we include only the most relevant ones. 
        \item [{\color{red}(6)}] The result of \citet{AhnYun2020} holds when $f$ satisfies the Polyak-{\L}ojasiewicz inequality, a generalization of strong convexity. We nevertheless specialize it to strong convexity for our comparison. The constant $G_0$ is defined as $G_0 \eqdef \sup_{x: f(x) \leq f(x_0)} \max_{i \in [n]} \norm{\nabla f_{i} (x)}$. Note that $\sigmaesc \leq G_0$. We show a better complexity for PL functions under bounded variance in Theorem~\ref{thm:rr-nonconvex}.
        \item [{\color{red}(7)}] This result is the first to show that RR and SO work with any $\gamma\le \frac{1}{L}$, but it asks for each $f_i$ to be strongly convex. The second result assumes that only $f$ is strongly convex.
      }
    \end{tablenotes}
  \end{threeparttable}
\end{table}

\section{Convergence theory}
\label{sec:conv-theory}

We will derive results for strongly convex, convex as well as non-convex  objectives. 
To compare between the performance of first-order methods, we define an $\e$-accurate solution as a point $\tilde{x} \in \R^d$ that satisfies (in expectation if $\tilde{x}$ is random)
\begin{align*}
  \norm{\nabla f(\tilde{x})} \leq \e, && \text { or } && \norm{\tilde{x} - x_\ast}^2 \leq \e, && \text { or } && f(\tilde{x}) - f(x_\ast) \leq \e
\end{align*}
for non-convex, strongly convex, and non-strongly convex objectives, respectively, and where $x_\ast$ is assumed to be a minimizer of $f$ if $f$ is convex. We then measure the performance of first-order methods by the number of individual gradients $\nabla f_{i} (\cdot)$ they access to reach an $\e$-accurate solution. 

Our first assumption is that the objective is bounded from below and smooth. This assumption is used in all of our results and is widely used in the literature.
\begin{assumption}
  \label{asm:f-smoothness}
The objective $f$ and the individual losses $f_{1}, \ldots, f_{n}$ are all $L$-smooth, i.e., their gradients are $L$-Lipschitz. Further, $f$ is lower bounded by some $f_\ast \in \R$. If $f$ is convex, we also assume the existence of a minimizer $x_\ast \in \R^d$.
\end{assumption}

Assumption~\ref{asm:f-smoothness} is necessary in order to obtain better convergence rates for RR compared to SGD, since without smoothness the SGD rate is optimal and cannot be improved \citep{Nagaraj2019}. The following quantity is key to our analysis and serves as an asymmetric distance between two points measured in terms of functions.
\begin{definition}
	For any $i$, the quantity $D_{f_i}(x, y)\eqdef f_i(x) - f_i(y) - \<\nabla f_i(y), x - y>$ is the \emph{Bregman divergence} between $x$ and $y$ associated with $f_i$.
\end{definition}
It is well-known that if $f_i$ is $L$-smooth and $\mu$-strongly convex, then for all $x, y \in \R^d$
\begin{equation}
  \label{eq:bregman-sc-smooth-properties}
  \textstyle \frac{\mu}{2}\|x-y\|^2\le D_{f_i}(x, y)\le \frac{L}{2}\|x-y\|^2,
\end{equation}
so each Bregman divergence is closely related to the Euclidian distance. Moreover, the difference between the gradients of a convex and $L$-smooth $f_i$ is related to its Bregman divergence by
\begin{equation}
  \label{eq:bregman-sc-grad-smooth-properties}
  \textstyle \sqn{\nabla f_i (x) - \nabla f_i (y)} \leq 2 L \cdot D_{f_i} (x, y).
\end{equation}

\subsection{Main result: strongly convex objectives}
\label{sec:strongly-convex}

Before we proceed to the formal statement of our main result, we need to present the central finding of our work. The analysis of many stochastic methods, including SGD, rely on the fact that the iterates converge to $x_\ast$ up to some noise. This is exactly where we part ways with the standard analysis techniques, since, it turns out, the intermediate iterates of shuffling algorithms converge to some other points. Given a permutation $\pi$, the real limit points are defined below,
\begin{align}
\textstyle	x_\ast^i \eqdef x_\ast - \gamma \sum \limits_{j=0}^{i-1} \nabla f_{\pi_{j}} (x_\ast), \qquad i=1,\dotsc, n-1. \label{eq:x_ast_i}
\end{align}
In fact, it is predicted by our theory and later validated by our experiments that within an epoch the iterates \emph{go away} from $x_\ast$, and closer to the end of the epoch they make a sudden comeback to $x_\ast$.

The second reason the vectors introduced in Equation~\eqref{eq:x_ast_i} are so pivotal is that they allow us to define a new notion of variance. Without it, there seems to be no explanation for why RR sometimes overtakes SGD from the very beginning of optimization process. We define it below.
\begin{definition}[Shuffling variance]
  \label{def:bregman-div-noise}
  Given a stepsize $\gamma>0$ and a random permutation $\pi$ of $\{ 1, 2, \ldots, n \}$, define $x_\ast^i$ as in \eqref{eq:x_ast_i}. Then, the shuffling variance is given by
  \begin{equation}\label{eq:bregman-div-noise}\textstyle  \sigmass \eqdef \max \limits_{i=1, \ldots, n-1} \left [ \frac{1}{\gamma}\ec{D_{f_{\pi_{i}}} (x_\ast^i, x_\ast)} \right ], \end{equation}
  where the expectation is taken with respect to the randomness in the permutation $\pi$.
\end{definition}

Naturally, $\sigmass$ depends on the functions $f_1, \ldots, f_n$, but, unlike SGD, it also depends in a non-trivial manner on the stepsize $\gamma$. The easiest way to understand the new notation is to compare it to the standard definition of variance used in the analysis of SGD. We argue that $\sigmass$ is the natural counter-part for the standard variance used in SGD. We relate both of them by the following upper and lower bounds:

\begin{proposition}
  \label{prop:shuffling-variance-normal-variance-bound}
  Suppose that each of $f_1, f_2, \ldots, f_{n}$ is $\mu$-strongly convex and $L$-smooth. Then
  $\frac{\gamma\mu n}{8}\sigmaesc^2
  \le \sigmass \leq \frac{\gamma L n}{4} \sigmaesc^2, $
  where $\sigmaesc^2 \eqdef \frac{1}{n} \sum_{i=1}^{n} \sqn{\nabla f_{i} (x_\ast)}$.
\end{proposition}

In practice, $\sigmass$ may be much closer to the lower bound than the upper bound; see Section~\ref{sec:experiments}. This leads to a dramatic difference in performance and provides additional evidence of the superiority of RR over SGD. The next theorem states how exactly convergence of RR depends on the introduced variance.

\begin{theorem}
  \label{thm:all-sc-rr-conv}
  Suppose that the functions $f_1,\dotsc, f_n$ are $\mu$-strongly convex and that Assumption~\ref{asm:f-smoothness} holds. Then for Algorithms~\ref{alg:rr} or~\ref{alg:so} run with a constant stepsize $\gamma \leq \frac{1}{L}$, the iterates generated by either of the algorithms satisfy
  \[  \ecn{x_{T} - x_\ast} \leq \br{1 - \gamma \mu}^{nT} \sqn{x_0 - x_\ast} + \frac{2\gamma \sigmass}{\mu}. \]
  \vspace{-2em}
\end{theorem}
\begin{proof}[Proof]
  The key insight of our proof is that the intermediate iterates $x_t^1, x_t^2, \ldots$ do not converge to $x_*$, but rather converge to the sequence $x_*^1, x_\ast^2, \ldots$ defined by \eqref{eq:x_ast_i}. Keeping this intuition in mind, it makes sense to study the following recursion:
  \begin{align}
    &\ec{\|x_t^{i+1}-x_*^{i+1}\|^2} \notag\\
    &=\ec{\|x_t^{i}-x_*^{i}\|^2-2\gamma\<\nabla f_{\pi_i}(x_t^i)-\nabla f_{\pi_i}(x_*), x_t^i - x_*^i>+\gamma^2\|\nabla f_{\pi_i}(x_t^i) - \nabla f_{\pi_i}(x_*)\|^2}.\label{eq:big89fg9h9d}
  \end{align}
  Once we have this recursion, it is useful to notice that the scalar product can be decomposed as
  \begin{align}
    \<\nabla f_{\pi_i}(x_t^i)-\nabla f_{\pi_i}(x_*), x_t^i - x_*^i>
    &= [f_{\pi_i}(x_*^i)-f_{\pi_i}(x_t^i)-\<\nabla f_{\pi_i}(x_t^i), x_*^i-x_t^i>] \notag \\
    & \quad + [f_{\pi_i}(x_t^i)-f_{\pi_i}(x_*)-\<\nabla f_{\pi_i}(x_*), x_t^i-x_*>] \notag  \\
    & \quad - [f_{\pi_i}(x_*^i)-f_{\pi_i}(x_*)-\<\nabla f_{\pi_i}(x_*), x_*^i-x_*>] \notag \\
    &= D_{f_{\pi_i}}(x_*^i, x_t^i)+D_{f_{\pi_i}}(x_t^i, x_*) - D_{f_{\pi_i}}(x_*^i, x_*).\label{eq:bi89gfdb09hff}
  \end{align}
  This decomposition is, in fact, very standard and is a special case of the so-called {\em three-point identity} \citep{Chen1993}. So, it should not be surprising that we use it. \\
  The rest of the proof relies on obtaining appropriate  bounds for the terms in the recursion. Firstly, we bound each of the three Bregman divergence terms appearing in \eqref{eq:bi89gfdb09hff}. By $\mu$-strong convexity of $f_i$, the first term in \eqref{eq:bi89gfdb09hff} satisfies
  \[
      \frac{\mu}{2}\|x_t^i - x_*^i\|^2 \overset{\eqref{eq:bregman-sc-smooth-properties}}{\leq} D_{f_{\pi_i}}(x_*^i, x_t^i)		,
  \]
  which we will use to obtain contraction.	The second term in \eqref{eq:bi89gfdb09hff} can be bounded via
  \[
    \frac{1}{2L}\|\nabla f_{\pi_i}(x_t^i) - \nabla f_{\pi_i}(x_*)\|^2
    \overset{\eqref{eq:bregman-sc-grad-smooth-properties}}{\le} D_{f_{\pi_i}}(x_t^i, x_*),
  \]
  which gets absorbed in the last term in the expansion of $\|x_t^{i+1}-x_*^{i+1}\|^2$.   The expectation of the third divergence term in 
  \eqref{eq:bi89gfdb09hff}   is trivially bounded as follows:
  \[ \ec{D_{f_{\pi_{i}}} (x_\ast^i, x_\ast) } \leq \max_{i=1, \ldots, n-1} \left [ \ec{D_{f_{\pi_{i}}} (x_\ast^i, x_\ast)} \right ] = \gamma\sigmass. \]
  Plugging these three bounds back into \eqref{eq:bi89gfdb09hff}, and the resulting inequality into \eqref{eq:big89fg9h9d}, we obtain
  \begin{align}
    \ec{\|x_t^{i+1}-x_*^{i+1}\|^2}
    &\le \ec{(1-\gamma\mu)\|x_t^i-x_*^i\|^2 - 2\gamma(1-\gamma L)D_{f_{\pi_i}}(x_t^i, x_*)} + 2 \gamma^2 \sigmass \nonumber \\
    \label{eq:thm_str_cvx_main-proof-1}
    &\le (1-\gamma\mu)\ec{\|x_t^i-x_*^i\|^2}+ 2 \gamma^2 \sigmass.
  \end{align}
  The rest of the proof is just solving this recursion, and is relegated to Section~\ref{sec:proof-of-thm-1} in the appendix.
\end{proof}

We show (\Cref{corr:all-sc-rr-conv} in the appendix) that by carefully controlling the stepsize, the final iterate of RR after $T$ epochs satisfies
\begin{equation}
  \label{eq:all-sc-rr-conv-bound-T}
 \textstyle  \ecn{x_{T} - x_\ast} = \mathcal{\tilde{O}} \br{ \exp\left(-\frac{\mu n T}{L} \right) \sqn{x_0 - x_\ast} + \frac{\kappa \sigmaesc^2}{\mu^2 n T^2} },
\end{equation}
where the $\tilde{\cO}(\cdot)$ notation suppresses absolute constants and polylogarithmic factors. Note that Theorem~\ref{thm:all-sc-rr-conv} covers both RR and SO, and for SO, \citet{Safran2020good} give almost the same {\em lower} bound. Stated in terms of the squared distance from the optimum, their lower bound is 
\[\textstyle  \ecn{x_{T} - x_\ast} = \Omega\br{\min \pbr{ 1, \frac{\sigmaesc^2}{\mu^2 n T^2} }}, \]
where we note that in their problem $\kappa = 1$. This translates to sample complexity $\mathcal{O}\br{ \sqrt{n} \sigmaesc/(\mu \sqrt{\e}) }$ for $\e \leq 1$\footnote{In their problem, the initialization point $x_0$ satisfies $\sqn{x_0 - x_\ast} \leq 1$ and hence asking for accuracy $\e > 1$ does not make sense.}. Specializing $\kappa = 1$ in \Cref{eq:all-sc-rr-conv-bound-T} gives the sample complexity of $\mathcal{\tilde{O}}\br{1 + \sqrt{n} \sigmaesc/(\mu \sqrt{\e}) }$, matching the optimal rate up to an extra iteration. More recently, \citet{Rajput2020} also proved a similar lower bound for RR. We emphasize that Theorem~\ref{thm:all-sc-rr-conv} is not only tight, but it is also the first convergence bound that applies to SO. Moreover, it also immediately works if one permutes once every few epochs, which interpolates between RR and SO mentioned by \citet{Rajput2020}.

\textbf{Comparison with SGD} To understand when RR is better than SGD, let us borrow a convergence bound for the latter. Several works have shown (e.g., see \citep{needell2014stochastic, Stich2019b}) that for any $\gamma\le \frac{1}{2L}$ the iterates of SGD satisfy
\[
	\textstyle  \ecn{x_{nT}^{\mathrm{SGD}} - x_\ast} \leq \br{1 - \gamma \mu}^{nT} \sqn{x_0 - x_\ast} + \frac{ 2\gamma \sigmaesc^2}{\mu}.
\]
Thus, the question as to which method will be faster boils down to which variance is smaller: $\sigmass$ or $\sigmaesc^2$. According to Proposition~\ref{prop:shuffling-variance-normal-variance-bound}, it depends on both $n$ and the stepsize. Once the stepsize is sufficiently small, $\sigmass$ becomes smaller than $\sigmaesc^2$, but this might not be true in general. Similarly, if we partition $n$ functions into $n/\tau$ groups, i.e., use minibatches of size $\tau$, then $\sigmaesc^2$ decreases as $\cO\br{1/\tau}$ and $\sigmass$ as $\cO\br{1/\tau^2}$, so RR can become faster even without decreasing the stepsize. We illustrate this later with numerical experiments.

While \Cref{thm:all-sc-rr-conv} requires each $f_i$ to  be strongly convex, we can also obtain results in the case where the individual strong convexity  assumption is replaced by convexity. However, in such a case, we need to use a smaller stepsize, as the next theorem shows.
\begin{theorem}
  \label{thm:only-f-sc-rr-conv}
  Suppose that each $f_i$ is convex, $f$ is $\mu$-strongly convex, and Assumption~\ref{asm:f-smoothness} holds. Then provided the stepsize satisfies $\gamma \leq \frac{1}{\sqrt{2} L n}$ the final iterate generated by Algorithms~\ref{alg:rr} or~\ref{alg:so} satisfies
  \[\textstyle  \ecn{x_T - x_\ast} \leq \br{ 1 - \frac{\gamma \mu n}{2} }^T \sqn{x_0 - x_\ast} + \gamma^2 \kappa n \sigmaesc^2. \]
\end{theorem}
It is not difficult to show that by properly choosing the stepsize $\gamma$, the guarantee given by Theorem~\ref{thm:only-f-sc-rr-conv} translates to a sample complexity of $\ctO\br{ \kappa n + \frac{\sqrt{\kappa n} \sigmaesc}{\mu \sqrt{\e}} }$, which matches the dependence on the accuracy $\e$ in Theorem~\ref{thm:all-sc-rr-conv} but with $\kappa (n-1)$ additional iterations in the beginning. For $\kappa = 1$, this translates to a sample complexity of $\ctO\br{n + \frac{\sqrt{n} \sigmaesc}{\mu \sqrt{\e}}}$ which is worse than the lower bound of \cite{Safran2020good} when $\e$ is large. In concurrent work, \citet{AhnYun2020} obtain in the same setting a complexity of $\ctO\br{1/\e^{1/\alpha} + \frac{\sqrt{n} G}{\mu \sqrt{\e}}}$ (for a constant $\alpha > 2$), which requires that each $f_i$ is Lipschitz and matches the lower bound only when the accuracy $\e$ is large enough that $1/\e^{1/\alpha} \leq 1$. Obtaining an optimal convergence guarantee for all accuracies $\e$ in the setting of Theorem~\ref{thm:only-f-sc-rr-conv} remains open.

\subsection{Non-strongly convex objectives}
\label{sec:weakly-convex}
We also make a step towards better bounds for RR/SO without any strong convexity at all and provide the following convergence statement.
\begin{theorem}
  \label{thm:weakly-convex-f-rr-conv}
  Let functions $f_1, f_2, \ldots, f_n$ be convex. Suppose that Assumption~\ref{asm:f-smoothness} holds. Then for Algorithm~\ref{alg:rr} or Algorithm~\ref{alg:so} run with a stepsize $\gamma \leq \frac{1}{\sqrt{2} L n}$, the average iterate $\hat{x}_{T} \eqdef \frac{1}{T} \sum_{j=1}^{T} x_j$ satisfies
\[ \textstyle \ec{f(\hat{x}_T) - f(x_\ast)} \le \frac{\sqn{x_0 - x_\ast}}{2 \gamma n T} + \frac{\gamma^2 L n \sigmaesc^2}{4}. \] 
\end{theorem}

Unfortunately, the theorem above relies on small stepsizes, but we still deem it as a valuable contribution, since it is based on a novel analysis. Indeed, the prior works showed that RR approximates a full gradient step, but we show that it is even closer to the implicit gradient step, see the appendix.

To translate the recursion in Theorem~\ref{thm:weakly-convex-f-rr-conv} to a complexity, one can choose a small stepsize and obtain (Corollary~\ref{corr:weakly-convex-f} in the appendix) the following bound for RR/SO:
\[ \textstyle  \ec{f(\hat{x}_T) - f(x_\ast)} = \mathcal{O}\br{ \frac{L \sqn{x_0 - x_\ast}}{T} + \frac{L^{1/3} \norm{x_0 - x_\ast}^{4/3} \sigmaesc^{2/3}}{n^{1/3} T^{2/3}} }. \]
\citet{Stich2019b} gives a convergence upper bound of $\mathcal{O}\br{ \frac{L \sqn{x_0 - x_\ast}}{nT} + \frac{\sigmaesc \norm{x_0 - x_\ast}}{\sqrt{n T}}}$ for SGD. Comparing upper bounds, we see that RR/SO beats SGD when the number of epochs satisfies $T \geq \frac{L^2 \sqn{x_0 - x_\ast} n}{\sigmaesc^2}$. To the best of our knowledge, there are no strict lower bounds in this setting. \citet{Safran2020good} suggest a lower bound of $\Omega\br{ \frac{\sigmaesc}{\sqrt{n T^3}} + \frac{\sigmaesc}{n T}} $ by setting $\mu$ to be small in their lower bound for $\mu$-strongly convex functions, however this bound may be too optimistic.

\subsection{Non-convex objectives}
\label{sec:non-convex}

For non-convex objectives, we formulate the following assumption on the gradients variance.
\begin{assumption}
  \label{asm:2nd-moment}
  There exist nonnegative constants $A, B \geq 0$ such that for any $x \in \R^d$ we have,
  \begin{equation}
    \label{eq:2nd-moment-bound}
 \textstyle    \frac{1}{n} \sum \limits_{i=1}^{n} \sqn{\nabla f_{i} (x) - \nabla f(x)} \leq 2 A \br{f(x) - f_\ast} + B^2.
  \end{equation}
\end{assumption}
Assumption~\ref{asm:2nd-moment} is quite general: if there exists some $G > 0$ such that $\norm{\nabla f_{i} (x)} \leq G$ for all $x \in \R^d$ and $i \in \{ 1, 2, \ldots, n \}$, then Assumption~\ref{asm:2nd-moment} is clearly satisfied by setting $A = 0$ and $B = G$. Assumption~\ref{asm:2nd-moment} also generalizes the uniformly bounded variance assumption commonly invoked in work on non-convex SGD, which is equivalent to~\eqref{eq:2nd-moment-bound} with $A=0$. Assumption~\ref{asm:2nd-moment} is a special case of the Expected Smoothness assumption of \citet{Khaled2020}, and it holds whenever each $f_{i}$ is smooth and lower-bounded, as the next proposition shows.

\begin{proposition}
  \label{prop:2nd-moment-bound}
  \citep[special case of Proposition 3]{Khaled2020} Suppose that $f_{1}, f_2, \ldots, f_n$ are lower bounded by $f_1^\ast, f_2^\ast, \ldots, f_n^\ast$ respectively and that Assumption~\ref{asm:f-smoothness} holds. Then there exist constants $A, B \geq 0$ such that Assumption~\ref{asm:2nd-moment} holds.
\end{proposition}

We now give our main convergence theorem for RR without assuming convexity.
\begin{theorem}
  \label{thm:rr-nonconvex}
  Suppose that Assumptions~\ref{asm:f-smoothness} and~\ref{asm:2nd-moment} hold. Then for Algorithm~\ref{alg:rr} run for $T$ epochs with a stepsize $\gamma \leq \min \br{ \frac{1}{2 L n}, \frac{1}{(A L^2 n^2 T)^{1/3}} }$ we have
  \[ 
 \textstyle    \min \limits_{t=0, \ldots, T-1} \ecn{\nabla f(x_t)} \leq \frac{12 \br{f(x_0) - f_\ast}}{\gamma n T} + 2 \gamma^2 L^2 n B^2. 
  \]
  If, in addition, $A=0$ and $f$ satisfies the Polyak-{\L}ojasiewicz inequality with $\mu>0$, i.e., $\|\nabla f(x)\|^2 \ge 2\mu(f(x)-f_*)$ for any $x\in\R^d$, then
  \[
	   \textstyle\ec{f(x_t)-f_*}
	  \le \br{1-\frac{\gamma\mu n}{2}}^T(f(x_0)-f_*) + \gamma^2\kappa L n B^2.
  \]
\end{theorem}
\textbf{Comparison with SGD.} From Theorem~\ref{thm:rr-nonconvex}, one can recover the complexity that we provide in Table~\ref{tab:conv-rates}, see Corollary~\ref{corr:rr-nonconvex} in the appendix. Let's ignore some constants not related to our assumptions and specialize to uniformly bounded variance. Then, the sample complexity of RR, 
$ K_{\mathrm{RR}} \geq \frac{ L \sqrt{n}}{\e^2} ( \sqrt{n}+ \frac{\sigma}{\e} ), $
becomes better than that of SGD,
$ K_{\mathrm{SGD}} \geq \frac{L}{\e^2} (1+ \frac{\sigma^2}{\e^2})$, whenever $\sqrt{n} \e \leq \sigma$.

\section{Experiments}
\begin{figure}[t]
	\noindent\makebox[\textwidth]{
	\includegraphics[width=0.33\linewidth]{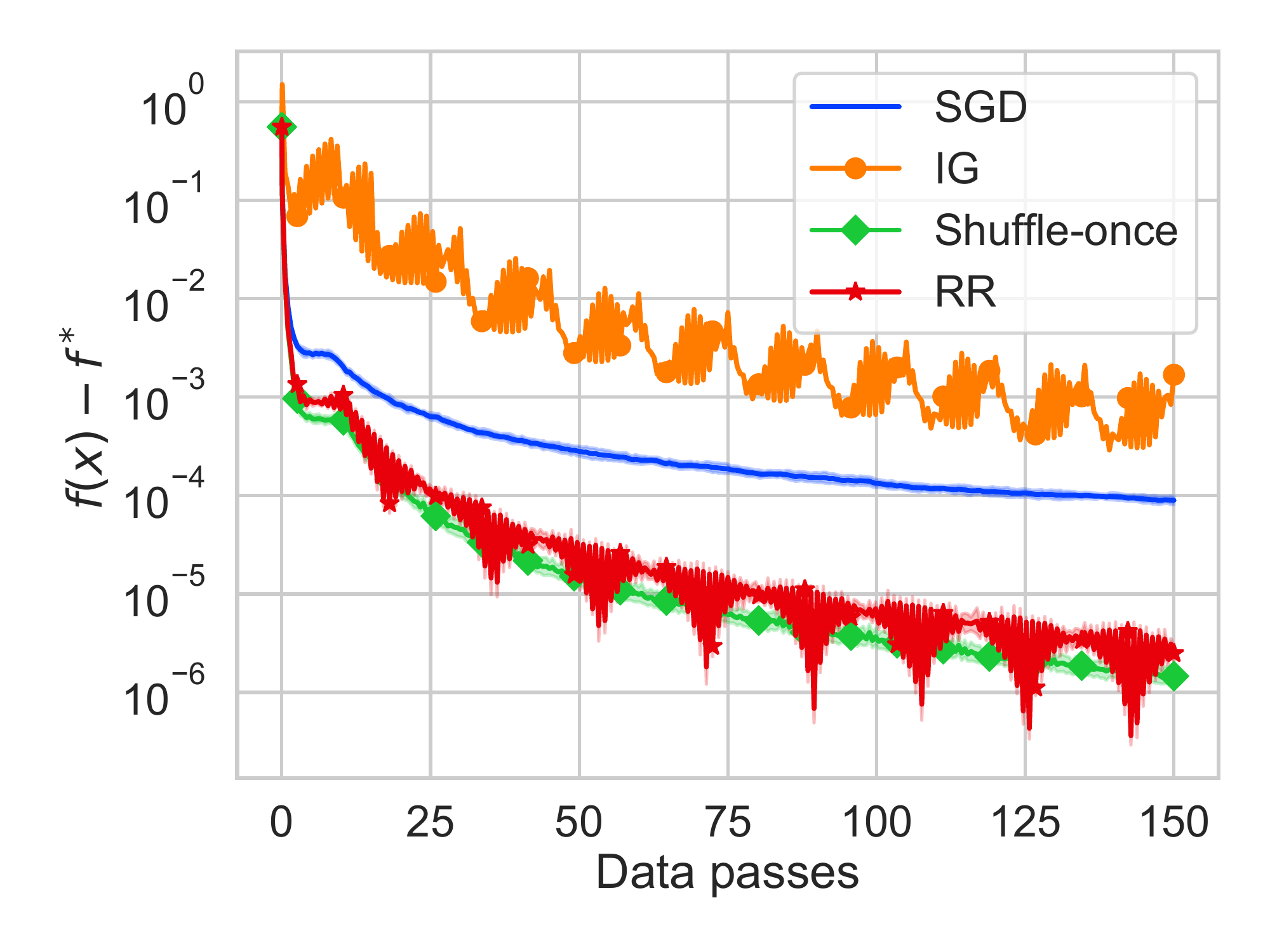}
	\includegraphics[width=0.33\linewidth]{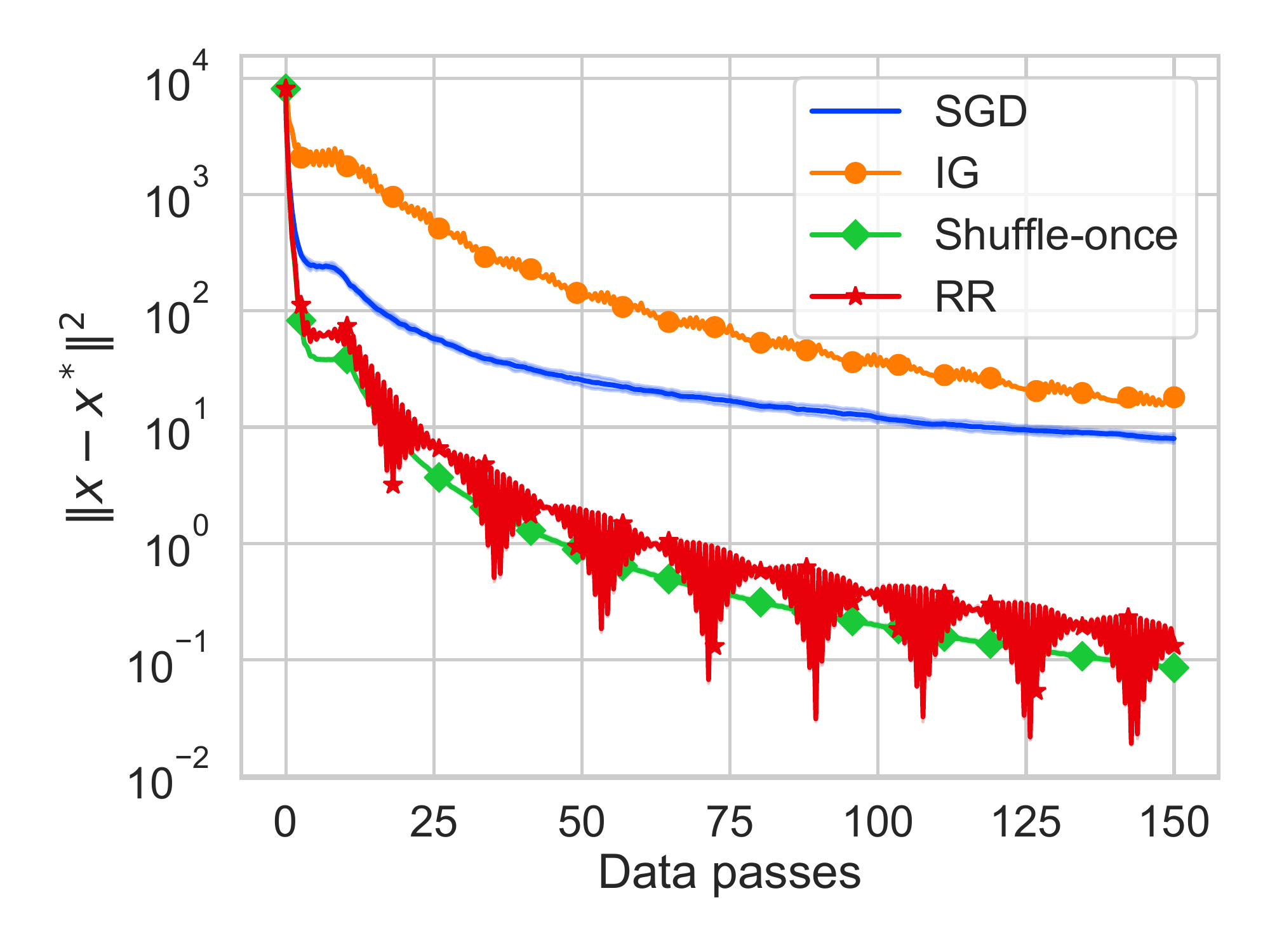}
	\includegraphics[width=0.33\linewidth]{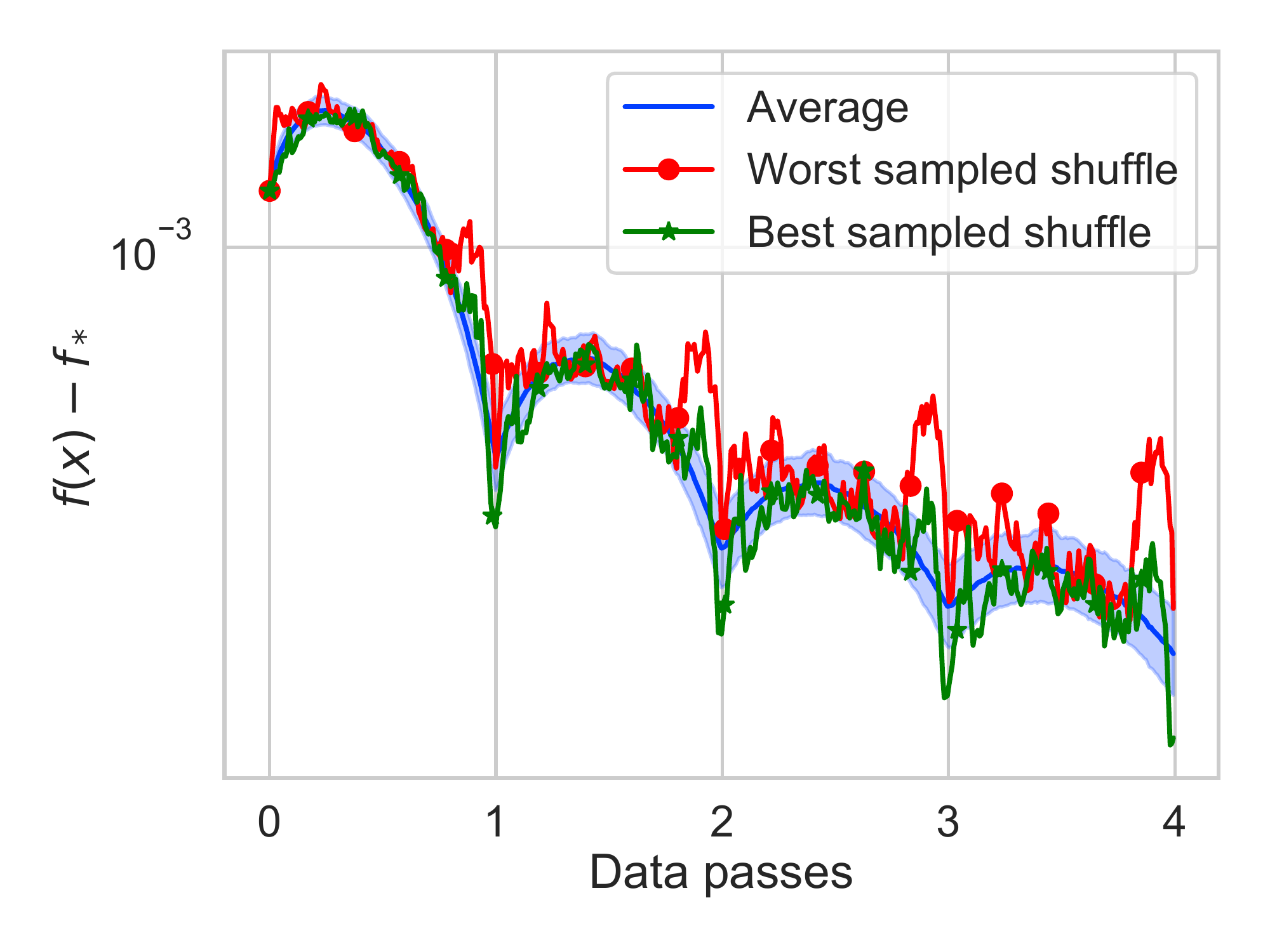}}
	\noindent\makebox[\textwidth]{
	\includegraphics[width=0.33\linewidth]{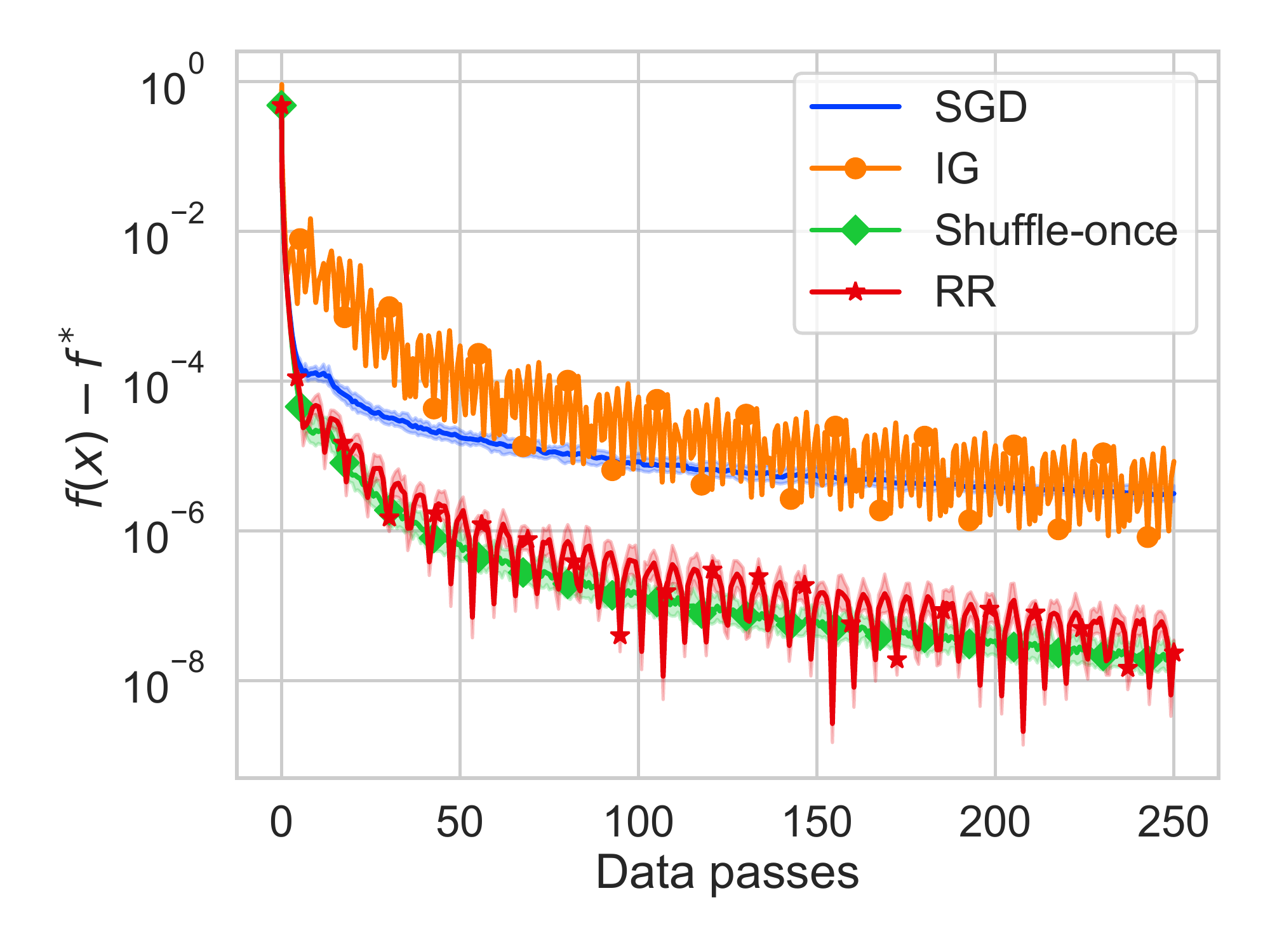}
	\includegraphics[width=0.33\linewidth]{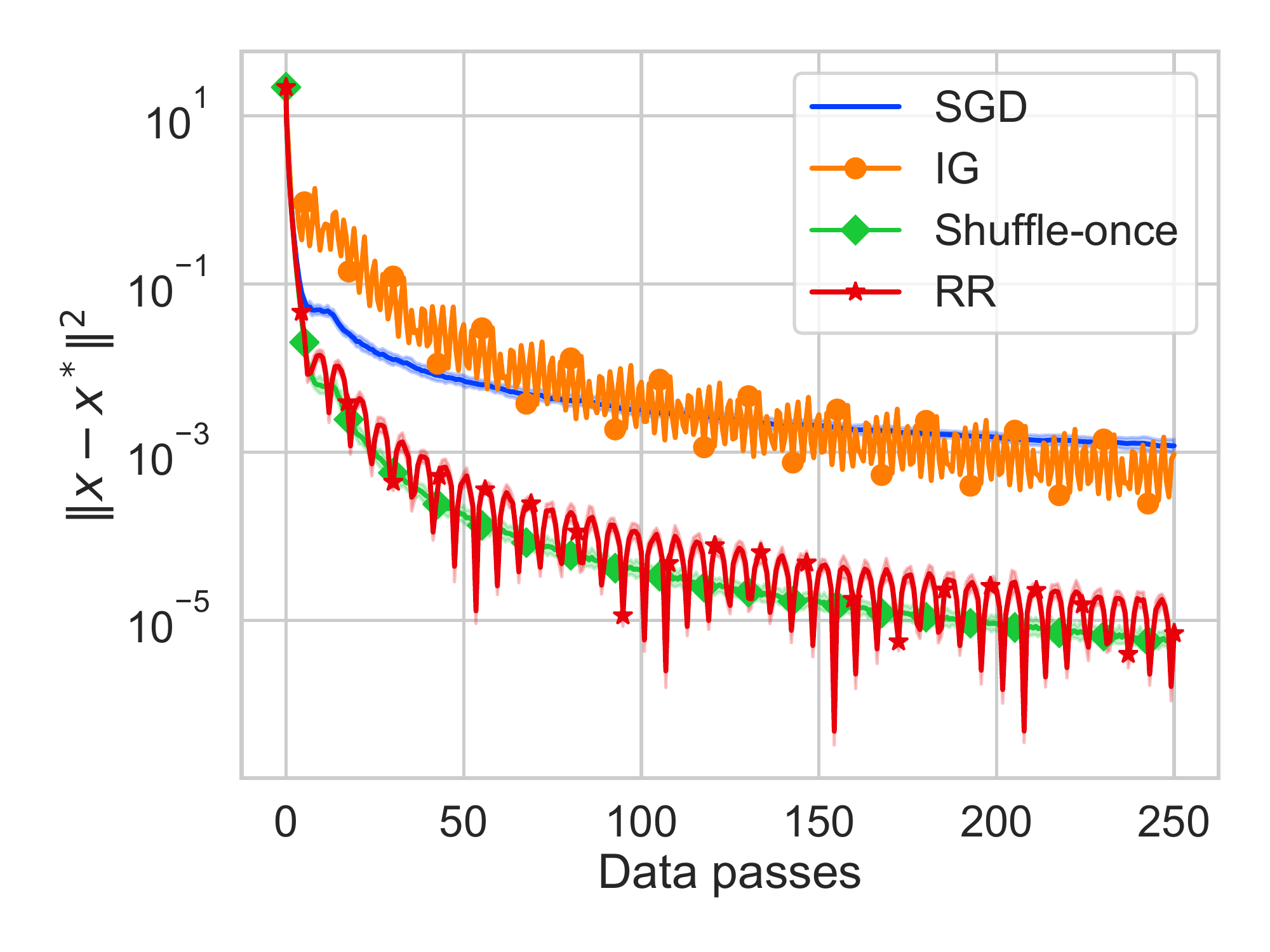}
	\includegraphics[width=0.33\linewidth]{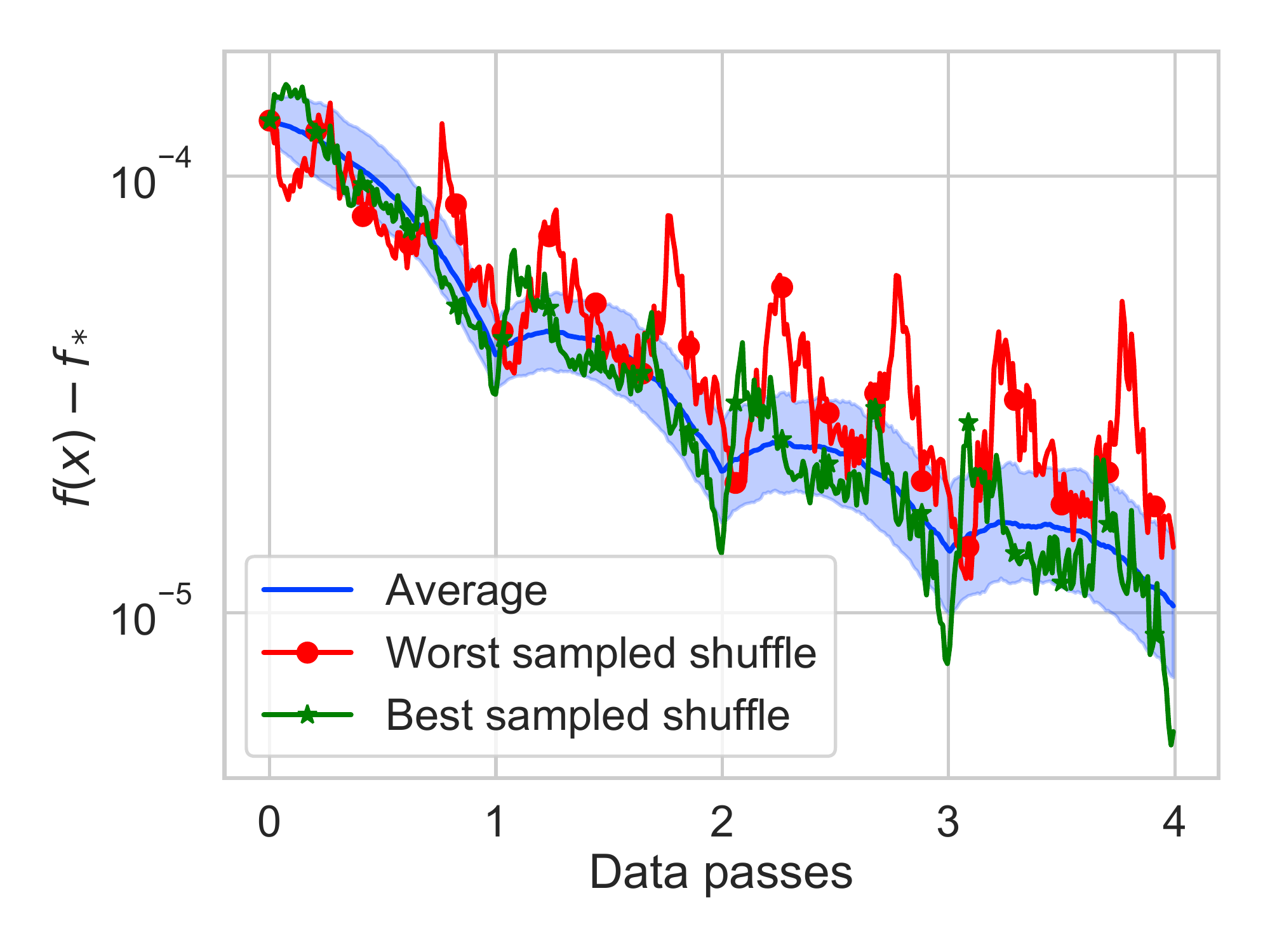}}
	\noindent\makebox[\textwidth]{
	\includegraphics[width=0.33\linewidth]{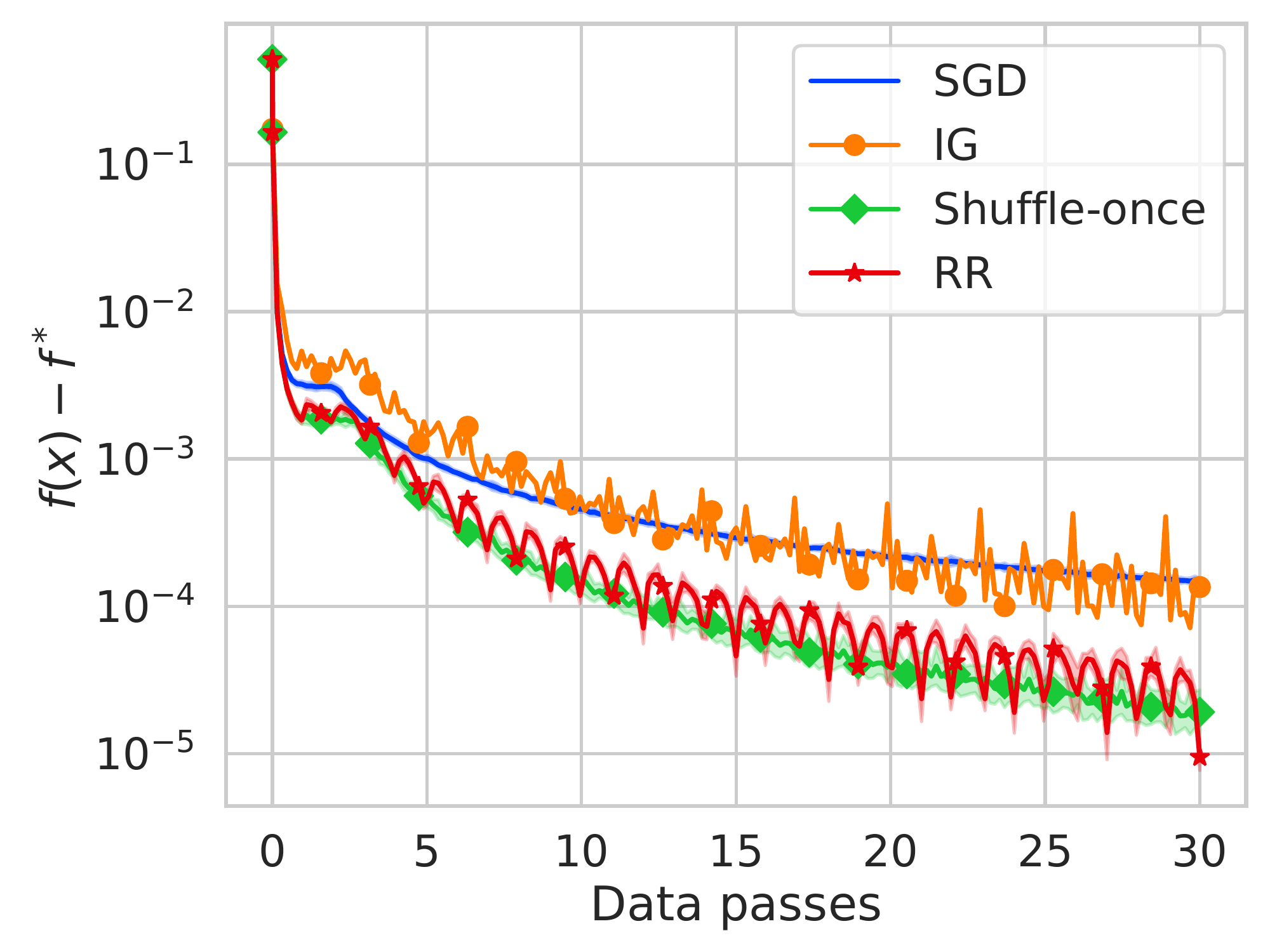}
	\includegraphics[width=0.33\linewidth]{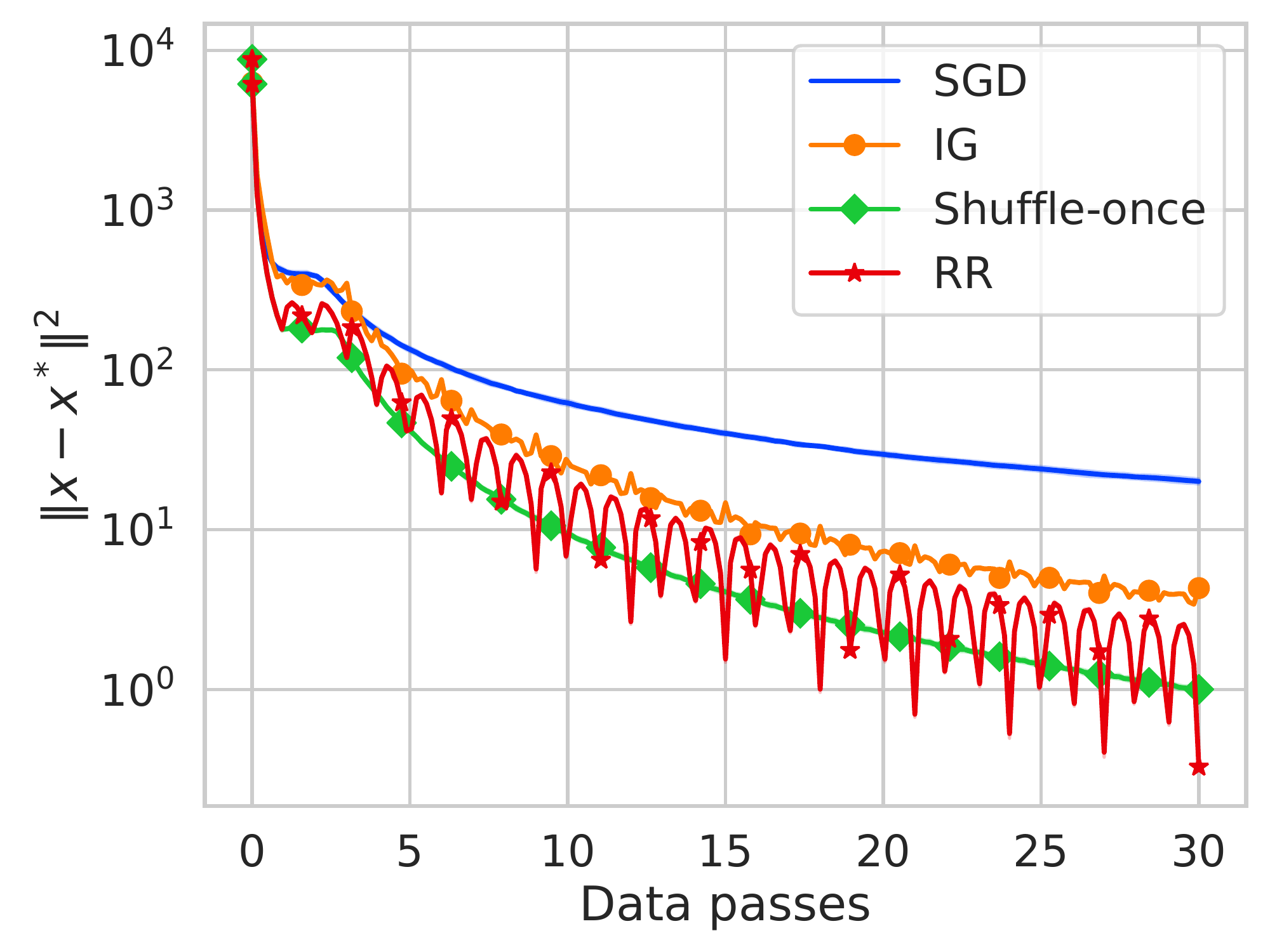}
	\includegraphics[width=0.33\linewidth]{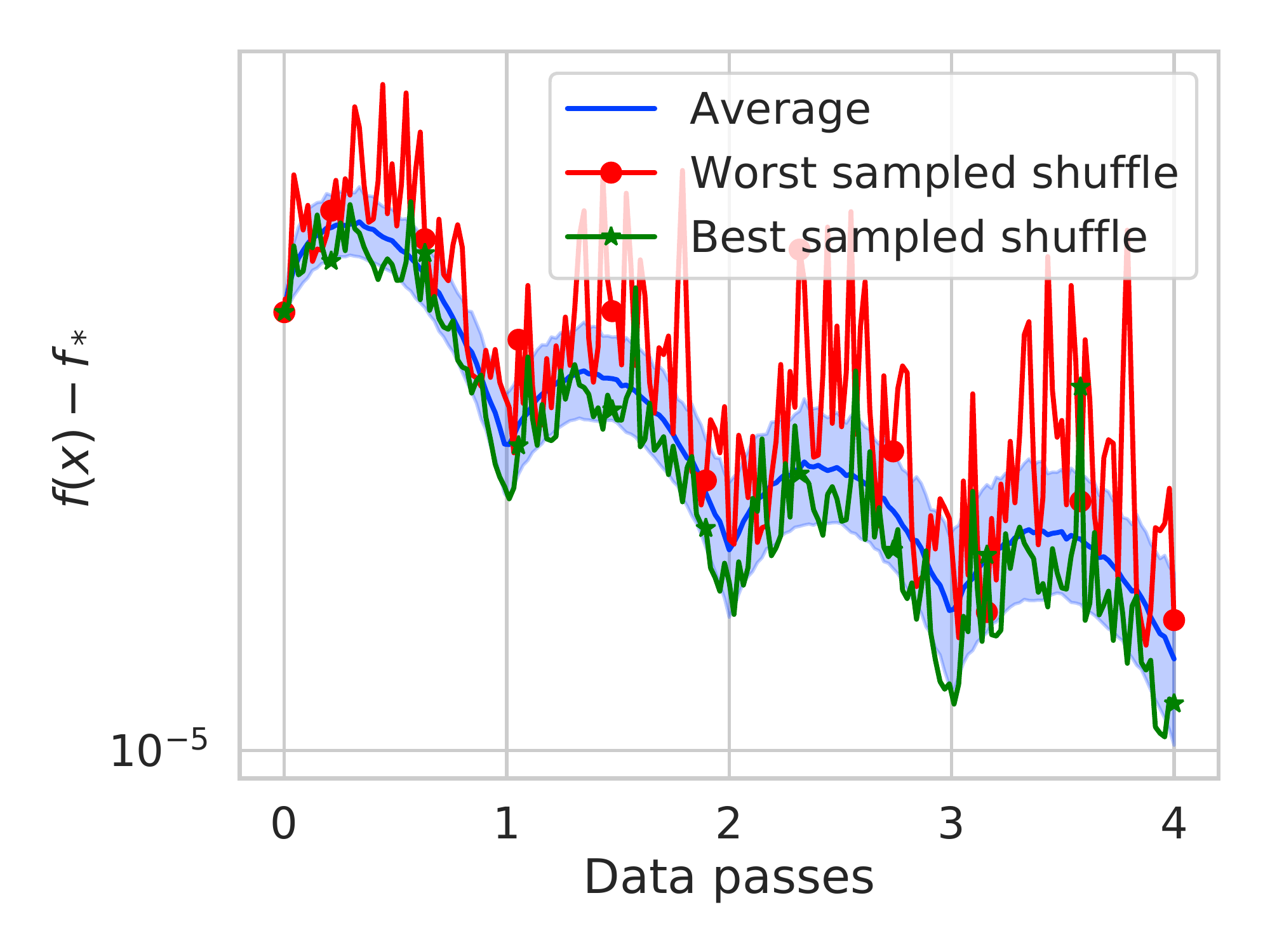}}
\caption{Top: {\tt real-sim} dataset ($N=72,309$; $d=20,958$), middle row: {\tt w8a} dataset ($N=49,749$; $d=300$), bottom: {\tt RCV1} dataset ($N=804,414$; $d=47,236$). Left: convergence of $f(x_t^i)$, middle column: convergence of $\|x_t^i-x_\ast\|^2$, right: convergence of SO with different permutations	.}
\label{fig:conv_plots}
\end{figure}

\label{sec:experiments}
We run our experiments on the $\ell_2$-regularized logistic regression problem given by
\[
\textstyle	\frac{1}{N}\sum \limits_{i=1}^N \br{-\big(b_i \log \big(h(a_i^\top x)\big) + (1-b_i)\log\big(1-h(a_i^\top x)\big)\big)}+\frac{\lambda}{2}\|x\|^2,
\]
where $(a_i, b_i)\in \R^d\times \{0, 1\}$, $i=1,\dotsc, N$ are the data samples and $h\colon t\to1/(1+e^{-t})$ is the sigmoid function. For better parallelism, we use minibatches of size 512 for all methods and datasets. 
We set $\lambda=L/\sqrt{N}$ and use stepsizes decreasing as $\cO(1/t)$. See the appendix for more details on the parameters used, implementation details, and reproducibility.

\textbf{Reproducibility.} Our code is provided at \href{https://github.com/konstmish/random_reshuffling}{https://github.com/konstmish/random\_reshuffling}. All used datasets are publicly available and all additional implementation details are provided in the appendix.

\textbf{Observations.} One notable property of all shuffling methods is that they converge with oscillations, as can be seen in Figure~\ref{fig:conv_plots}. There is nothing surprising about this as the proof of our Theorem~\ref{thm:all-sc-rr-conv} shows that the intermediate iterates converge to $x_*^i$ instead of $x_*$. It is, however, surprising how striking the difference between the intermediate iterates within one epoch can be.

Next, one can see that SO and RR converge almost the same way, which is in line with Theorem~\ref{thm:all-sc-rr-conv}. On the other hand, the contrast with IG is dramatic, suggesting existence of bad permutations. The probability of getting such a permutation seems negligible; see the right plot in Figure~\ref{fig:variance}.

Finally, we remark that the first two plots in Figure~\ref{fig:variance} demonstrate the importance of the new variance introduced in Definition~\ref{def:bregman-div-noise}. The upper and lower bounds from Proposition~\ref{prop:shuffling-variance-normal-variance-bound} are depicted in these two plots and one can observe that the lower bound is often closer to the actual value of $\sigmass$ than the upper bound. And the fact that $\sigmass$ very quickly becomes smaller than $\sigmaesc^2$ explains why RR often outperforms SGD starting from  early iterations.

\begin{figure}[h]
	\noindent\makebox[\textwidth]{
	\includegraphics[width=0.33\linewidth]{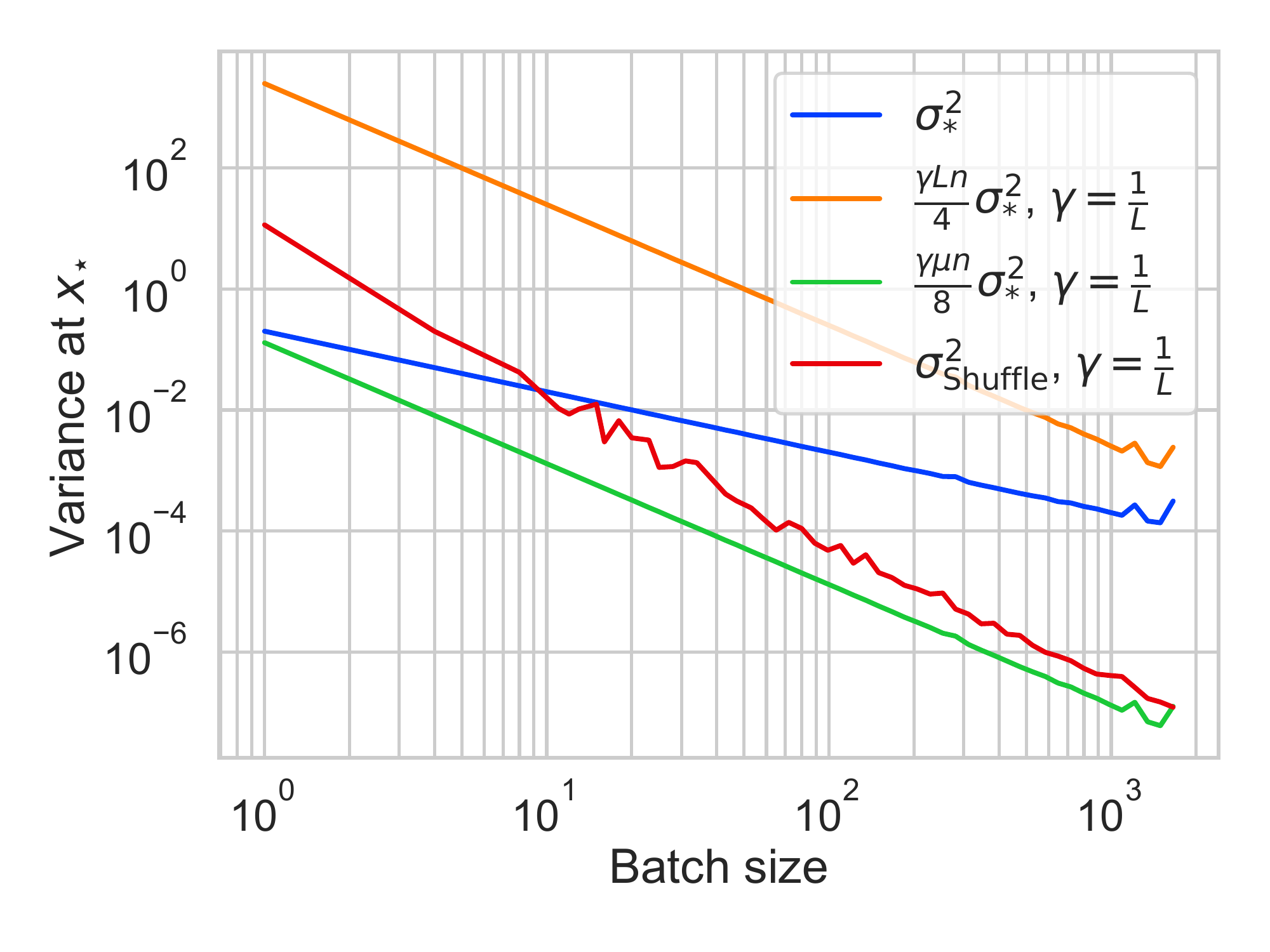}
	\includegraphics[width=0.33\linewidth]{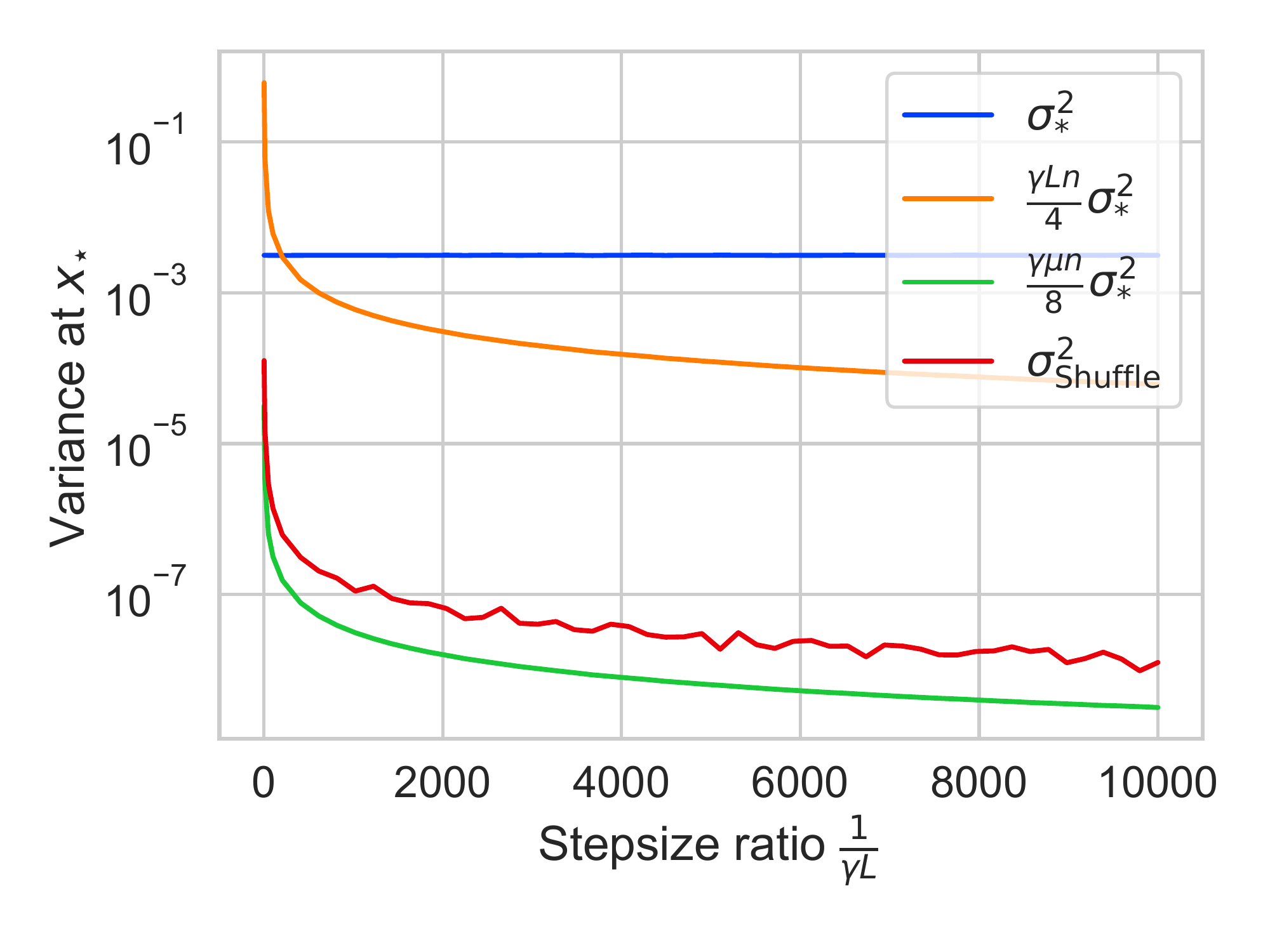}
	\includegraphics[width=0.33\linewidth]{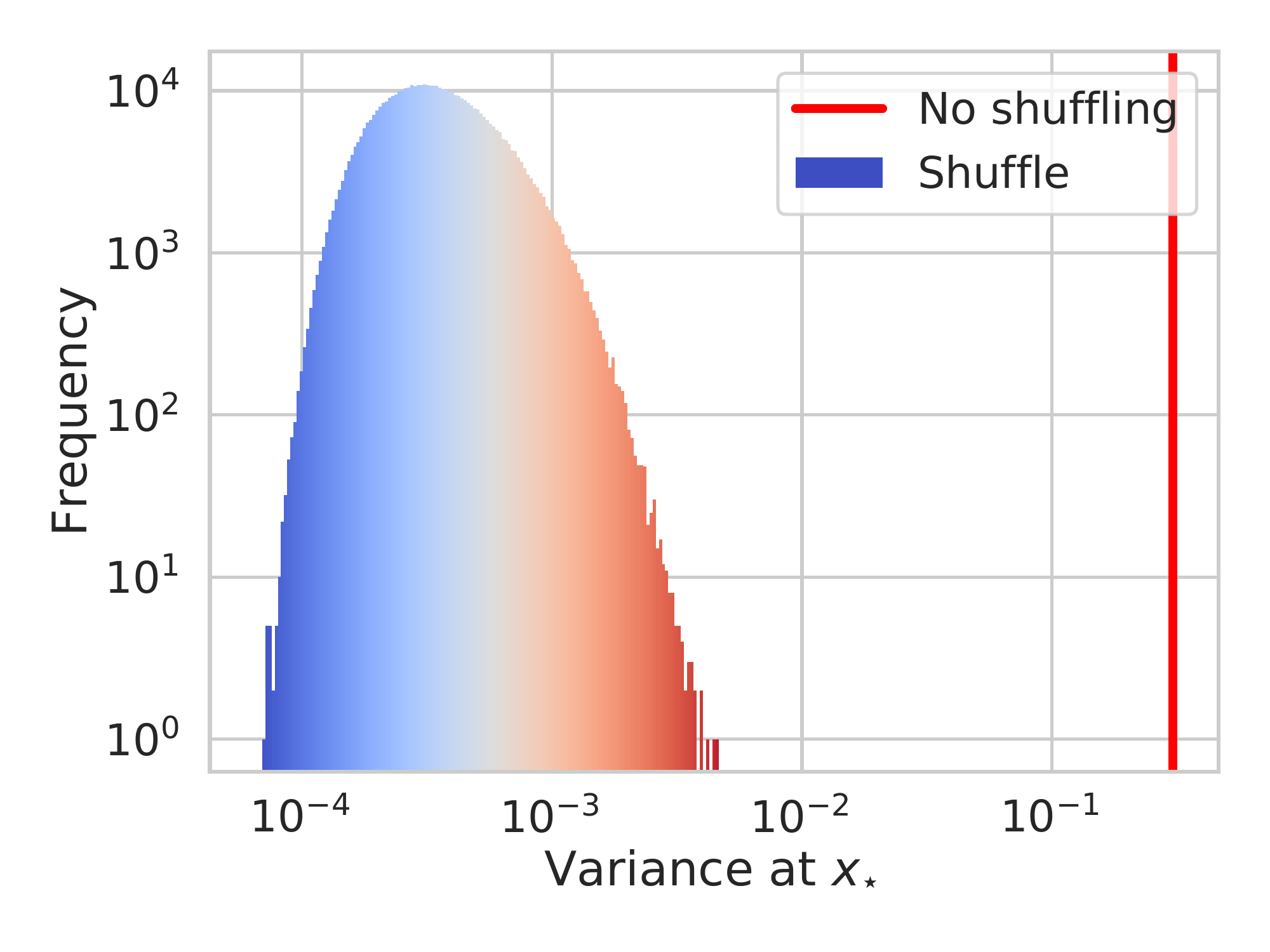}
	}
\caption{Estimated variance at the optimum, $\sigmass$ and $\sigmaesc^2$, for the {\tt w8a} dataset. Left: the values of variance for different minibatch sizes with $\gamma=\nicefrac{1}{L}$. Middle: variance with fixed minibatch size 64 for different $\gamma$, starting with $\gamma=\nicefrac{1}{L}$ and ending with $\gamma=\nicefrac{10^{-4}}{L}$. Right: the empirical distribution of $\sigmass$ for $500,000$ sampled permutations with $\gamma=\nicefrac{1}{L}$ and minibatch size 64.}
	\label{fig:variance}
\end{figure}

\clearpage
\section*{Broader Impact}

Our contribution is primarily theoretical. Moreover, we study methods that are already in use in practice, but are notoriously hard to analyze. We believe we have made a breakthrough in this area by developing new and remarkably simple proof techniques, leading to  sharp bounds. This, we hope, will inspire other researchers to apply and further develop our techniques to other contexts and algorithms. These applications may one day push the state of the art in practice for existing or new supervised machine learning applications, which may then have broader impacts. Besides this, we do not expect any direct or short term  societal consequences.

\begin{ack}
  Ahmed Khaled acknowledges internship support from the Optimization and Machine Learning Lab led by Peter Richt\'{a}rik at KAUST.
\end{ack}

\bibliographystyle{plainnat}
\bibliography{biblio}

\clearpage
\part*{Appendix}
\tableofcontents

\clearpage

\section{Additional experiment details}
\textbf{Objective properties.} To better correspond to the theoretical setting of our main result, we use $\ell_2$ regularization in every element of the finite-sum. To obtain minibatches for the RR, SO and IG we permute the dataset and then split it into $n=\lceil \frac{N}{\tau}\rceil$ groups of sizes $\tau, \dotsc, \tau, N-\tau \br{\lceil\frac{N}{\tau}\rceil - 1}$. In other words, the first $n-1$ groups are of size $\tau$ and the remaining samples go to the last group. For SO and IG, we split the data only once, and for RR, we do this at the beginning of each epoch. The permutation of samples used in IG is the one in which the datasets are stored online. The smoothness constant of the sum of logistic regression losses admits a closed form expression $L_f=\frac{1}{4N}\|A\|^2+\lambda$. The individual losses are $L_{\max}$-smooth with $L_{\max}=\max_{i=1,\dotsc, n} \| a_i\|^2+\lambda$. 

\textbf{Stepsizes.} For all methods in Figure~\ref{fig:conv_plots}, we keep the stepsize equal to $\frac{1}{L}$ for the first $k_0=\lfloor K/40 \rfloor$ iterations, where $K$ is the total number of stochastic steps. This is important to ensure that there is an exponential convergence before the methods reach their convergence neighborhoods~\citep{Stich2019b}. After the initial $k_0$ iterations, the stepsizes used for RR, SO and IG were chosen as $\gamma_{k}=\min\left\{\frac{1}{L}, \frac{3}{\mu \max\{1, k-k_0\}}\right\}$ and for SGD as $\gamma_{k}=\min\left\{\frac{1}{L}, \frac{2}{\mu \max\{1, k-k_0\}}\right\}$. Although these stepsizes for RR are commonly used in practice~\citep{Bottou2009}, we do not analyze them and leave decreasing-stepsize analysis for future work. We also note that although $L$ is generally not available, it can be estimated using empirically observed gradients \citep{malitsky2019adaptive}. For our experiments, we estimate $L$ of minibatches of size $\tau$ using the closed-form expressions from Proposition~3.8 in \citep{Gower2019} as $L\le \frac{n(\tau -1)}{\tau(n-1)}L_f + \frac{n-\tau}{\tau(n-1)}L_{\max}$. The confidence intervals in Figure~\ref{fig:conv_plots} are estimated using 20 random seeds.

For the experiments in Figure~\ref{fig:variance}, we estimate the expectation from~\eqref{eq:bregman-div-noise} with 20 permutations, which provides sufficiently stable estimates. In addition, we use $L=L_{\max}$ (instead of using the batch smoothness of \citep{Gower2019}) as the plots in this figure use different minibatch sizes and we want to isolate the effect of reducing the variance by minibatching from the effect of changing $L$.

\textbf{SGD implementation.} For SGD, we used two approaches to minibatching. In the first, we sampled $\tau$ indices from $\{1,\dotsc, N\}$ and used them to form the minibatch, where $N$ is the total number of data samples. In the second approach, we permuted the data once and then at each iteration, we only sampled one index $i$ and formed the minibatch from indices $i, (i+1) \mod N,\dotsc, (i+\tau-1) \mod N$. The latter approach is much more cash-friendly and runs significantly faster, while the iteration convergence was the same in our experiments. Thus, we used the latter option to produce the final plots.

For all plots and methods, we use zero initialization, $x_0=(0,\dotsc, 0)^\top \in \R^d$. We obtain the optimum, $x_*$, by running Nesterov's accelerated gradient method until it reaches machine precision. The plots in the right column in Figure~\ref{fig:variance} were obtained by initializing the methods at an intermediate iterate of Nesterov's method, and we found the average, best and worst results by sampling 1,000 permutations.

\newpage
\section{Basic facts and notation}

\subsection{Basic identities and  inequalities}

For any two vectors $a, b \in \R^d$, we have
\begin{equation}
    \label{eq:square-decompos}
    2 \ev{a, b} = \sqn{a} + \sqn{b} - \sqn{a - b}.
\end{equation}
As a consequence of \eqref{eq:square-decompos}, we get
\begin{equation}
  \label{eq:sqnorm-triangle-2}
  \sqn{a + b} \leq 2 \sqn{a} + 2 \sqn{b}.
\end{equation}

\paragraph{Convexity, strong convexity and smoothness.} A differentiable function $h:\R^d\to \R$ is called $\mu$-convex if for some $\mu \geq 0$ and for all $x, y \in \R^d$, we have
\begin{equation}
  \label{eq:mu-convexity}
  h(x) + \ev{\nabla h(x), y - x} + \frac{\mu}{2} \sqn{y - x} \leq h(y).
\end{equation}
If $h$ satisfies \eqref{eq:mu-convexity} with $\mu > 0$, then we say that $h$ is $\mu$-strongly convex, and if $\mu = 0$ then we say $h$ is convex. A differentiable function $h:\R^d\to \R$ is called $L$-smooth if for some $L\geq 0$ and for all $x, y \in \R^d$, we have
\begin{equation}
  \label{eq:nabla-Lip}
  \norm{\nabla h(x) - \nabla h(y)} \leq L \norm{x - y}.
\end{equation}
A useful consequence of $L$-smoothness is the inequality
\begin{equation}
  \label{eq:L-smoothness}
  h(x) \leq h(y) + \ev{\nabla h(y), x - y} + \frac{L}{2} \sqn{x - y},
\end{equation}
holding for all $x,y\in \R^d$. If $h$ is $L$-smooth and lower bounded by $h_\ast$, then
\begin{equation}
  \label{eq:grad-bound}
  \sqn{\nabla h(x)} \leq 2 L \br{h(x) - h_\ast}.
\end{equation}

For any convex and $L$-smooth function $h$ it holds
\begin{align}
	D_h(x, y)\ge \frac{1}{2L}\norm{\nabla h(x) - \nabla h(y)}^2. \label{eq:bregman_lower_smooth}
\end{align}

\paragraph{Jensen's inequality and consequences.} For a convex function $h: \R^d \to \R$ and any vectors $y_1,\dots,y_n \in \R^d$,  Jensen's inequality states that
\[
  \label{eq:jensen}
  h\br{\frac{1}{n} \sum_{i=1}^{n} y_i} \leq \frac{1}{n} \sum_{i=1}^{n} h(y_i).
\]
Applying this to the squared norm, $h(y) = \sqn{y}$, we get
\begin{equation}
  \label{eq:sqnorm-jensen}
  \sqn{ \frac{1}{n} \sum_{i=1}^{n} y_i } \leq \frac{1}{n} \sum_{i=1}^{n} \sqn{y_i}.
\end{equation}
After multiplying  both sides of \eqref{eq:sqnorm-jensen} by $n^2$, we get
\begin{equation}
  \label{eq:sqnorm-sum-bound}
  \sqn{\sum_{i=1}^{n} y_i} \leq n \sum_{i=1}^{n} \sqn{y_i}.
\end{equation}

\paragraph{Variance decomposition.}
We will use the following decomposition that holds for any random variable $X$ with $\ec{\norm{X}^2}<+\infty$,
\begin{align}
	\ec{\norm{X}^2}=\norm{\ec{X}}^2 + \ec{\norm{X-\ec{X}}^2}. \label{eq:rv_moments}
\end{align}
We will make use of the particularization of \eqref{eq:rv_moments} to the discrete case: let $y_{1}, \ldots, y_{m} \in \R^d$ be given vectors and let $\bar{y} = \frac{1}{m} \sum_{i=1}^{m} y_i$ be their average. Then,
\begin{equation}
  \label{eq:variance-decomp}
  \frac{1}{m} \sum_{i=1}^{m} \sqn{y_{i}} = \sqn{ \bar{y} } + \frac{1}{m} \sum_{i=1}^{m} \sqn{y_i - \bar{y}}.
\end{equation}

\subsection{Notation}
We define the epoch total gradient $g_t$ as
\[ g_t \eqdef \sum_{i=0}^{n-1} \nabla f_{\prm{i}} (x_t^i). \]
We define the variance of the local gradients from their average at a point $x_t$ as
\[ \sigma_{t}^2 \eqdef \frac{1}{n} \sum_{j=1}^{n} \sqn{\nabla f_{j} (x_t) - \nabla f(x_t)}. \]
By $\et{\cdot}$ we denote the expectation conditional on all information prior to iteration $t$, including $x_t$. To avoid issues with the special case $n=1$, we use the convention $0/0=0$.
A summary of the notation used in this work is given in Table~\ref{tab:notation}.

\begin{table}[t]
  \centering
  \caption{Summary of key notation used in the paper.}
  \label{tab:notation}
  \begin{tabular}{@{}cl@{}}
  \toprule
  Symbol               & Description                                                                                                                                    \\ \midrule
  $x_t$                & The iterate used at the start of epoch $t$.                                                                                                    \\ \midrule
  $\pi$ &
    \begin{tabular}[l]{@{}l@{}}A permutation $\pi = \br{ \prm{0}, \prm{1}, \ldots, \prm{n-1} }$ of $\{ 1, 2, \ldots, n \}$. \\ Fixed for Shuffle-Once and resampled every epoch for Random Reshuffling.\end{tabular} \\ \midrule
  $\gamma$             & The stepsize used when taking descent steps in an epoch.                                                                                       \\ \midrule
  $x_t^i$              & The current iterate after $i$ steps in epoch $t$, for $0 \leq i \leq n$.                                                                       \\ \midrule
  $g_t$                & The sum of gradients used over epoch $t$ such that $x_{t+1}=x_t-\gamma g_t$.                                                                                                      \\ \midrule
  $\sigma_{t}^2$       & The variance of the individual loss gradients from the average loss at point $x_t$.                                                            \\ \midrule
  $A, B$ & Assumption~\ref{asm:2nd-moment} constants.                                                                                                  \\ \midrule
  $L$                  & The smoothness constant of $f$ and $f_{1}, f_{2}, \ldots, f_{n}$.                                                                              \\ \midrule
  $\mu$              & The strong convexity constant (for strongly convex objectives).                                                                                                                   \\ \midrule
  $\kappa$           & The condition number $\kappa \eqdef L/\mu$ for strongly convex objectives. \\ \midrule
  $\delta_{t}$         & \begin{tabular}[c]{@{}c@{}}Functional suboptimality, $\delta_{t} = f(x_t) - f_\ast$, where $f_\ast= \inf_{x} f(x)$.\end{tabular} \\ \midrule
  $r_{t}$              & \begin{tabular}[c]{@{}c@{}}The squared iterate distance from an optimum for convex losses \\ $r_{t} = \sqn{x_t - x_\ast}$.\end{tabular}       \\ \midrule
  $\et{\cdot}$              & \begin{tabular}[c]{@{}c@{}}Expectation conditional on the history of the algorithm prior to timestep $t$,\\ including $x_{t}$.\end{tabular}    \\ \bottomrule
  \end{tabular}
\end{table}

\clearpage
\section{A lemma for sampling without replacement}
The following algorithm-independent lemma characterizes the variance of sampling a number of vectors from a finite set of vectors, without replacement. It is a key ingredient in our results on the convergence of the RR and SO methods.

\begin{lemma}\label{lem:sampling_wo_replacement}
	Let $X_1,\dotsc, X_n\in \R^d$ be fixed vectors, $\overline X\eqdef \frac{1}{n}\sum_{i=1}^n X_i$ be their average and $\sigma^2 \eqdef \frac{1}{n}\sum_{i=1}^n \norm{X_i-\overline X}^2$ be the population variance. Fix any $k\in\{1,\dotsc, n\}$, let $X_{\pi_1}, \dotsc X_{\pi_k}$ be sampled uniformly without replacement from $\{X_1,\dotsc, X_n\}$ and $\overline X_\pi$ be their average. Then, the sample average and variance are given by
	\begin{align}
		\ec{\overline X_\pi}=\overline X, && \ec{\norm{\overline X_{\pi} - \overline X}^2}= \frac{n-k}{k(n-1)}\sigma^2. \label{eq:sampling_wo_replacement}
	\end{align}
\end{lemma}
\begin{proof}
	The first claim follows by linearity of expectation and uniformity of  sampling:
	\[
		\ec{\overline X_\pi} 
		= \frac{1}{k}\sum_{i=1}^k \ec{X_{\pi_i}}
		= \frac{1}{k}\sum_{i=1}^k \overline X
		= \overline X.
	\]
	To prove the second claim, let us first establish that the identity $\mathrm{cov}(X_{\pi_i}, X_{\pi_j})=-\frac{\sigma^2}{n-1}$ holds  for any $i\neq j$. Indeed,
	\begin{align*}
		\mathrm{cov}(X_{\pi_i}, X_{\pi_j})
		&= \ec{ \ev{X_{\pi_i} - \overline X, X_{\pi_j} - \overline X}}
		= \frac{1}{n(n-1)}\sum_{l=1}^n\sum_{m=1,m\neq l}^n\ev{X_l - \overline X, X_m - \overline X} \\
		&= \frac{1}{n(n-1)}\sum_{l=1}^n\sum_{m=1}^n\ev{X_l - \overline X, X_m - \overline X} - \frac{1}{n(n-1)}\sum_{l=1}^n \norm{X_l - \overline X}^2 \\
		&= \frac{1}{n(n-1)}\sum_{l=1}^n \ev{X_l - \overline X, \sum_{m=1}^n(X_m - \overline X)} - \frac{\sigma^2}{n-1} \\
		&=-\frac{\sigma^2}{n-1}.
	\end{align*}
This identity helps us to establish the formula for sample variance:
	\begin{align*}
		\ecn{\overline X_{\pi} - \overline X}
		&= \frac{1}{k^2} \sum_{i=1}^k\sum_{j=1}^k \mathrm{cov}(X_{\pi_i}, X_{\pi_j}) \\
    &= \frac{1}{k^2}\ec{\sum_{i=1}^k \sqn{X_{\pi_i} - \overline X}} + \sum_{i=1}^k\sum_{j=1,j\neq i}^{k} \mathrm{cov}(X_{\pi_i}, X_{\pi_j})  \\
		&=\frac{1}{k^2}\left(k\sigma^2 - k(k-1)\frac{\sigma^2}{n-1}\right)
    = \frac{n-k}{k(n-1)}\sigma^2.
    \qedhere
	\end{align*}
\end{proof}

\clearpage
\section{Proofs for convex objectives (Sections~\ref{sec:strongly-convex} and~\ref{sec:weakly-convex})}

\subsection{Proof of Proposition~\ref{prop:shuffling-variance-normal-variance-bound}}
\begin{proof}
  Let us start with the upper bound. Fixing any $i$ such that $1 \leq i \leq n-1$, we have $i(n-i)\le \frac{n^2}{4}\le \frac{n(n-1)}{2}$ and using smoothness and Lemma~\ref{lem:sampling_wo_replacement} leads to
	\begin{eqnarray*}
		\ec{D_{f_{\pi_i}}(x_*^i, x_*)}
		\overset{\eqref{eq:L-smoothness}}{\le}  \frac{L}{2}\ec{\|x_*^i-x_*\|^2} 
		&=& \frac{L}{2}\ec{\sqn{\sum_{k=0}^{i-1}\gamma \nabla f_{\pi_k}(x_*)}} \\
		&\overset{\eqref{eq:sampling_wo_replacement}}{=} & \frac{\gamma^2 L i(n-i)}{2(n-1)}\sigma_*^2 \\
    &\le & \frac{\gamma^2 L n}{4}\sigma_*^2.
  \end{eqnarray*}
  To obtain the upper bound, it remains to take  maximum with respect to $i$ on both sides and divide by $\gamma$.  To prove the lower bound, we use strong convexity and the fact that  $\max_i i(n-i)\ge \frac{n(n-1)}{4}$ holds for any integer $n$. Together, this leads to
  \[
  	\max_i\ec{D_{f_{\pi_i}}(x_*^i, x_*)}
		\overset{\eqref{eq:mu-convexity}}{\ge} \max_i\frac{\mu}{2}\ec{\|x_*^i-x_*\|^2} 
    = \max_i\frac{\gamma^2 \mu i(n-i)}{2(n-1)}\sigma_*^2 \ge \frac{\gamma^2 \mu n}{8}\sigma_*^2,
  \]
  as desired.    
\end{proof}

\subsection{Proof Remainder for Theorem~\ref{thm:all-sc-rr-conv}}
\label{sec:proof-of-thm-1}
\begin{proof}
  We start from \eqref{eq:thm_str_cvx_main-proof-1} proved in the main text:
	\[
		\ec{\|x_t^{i+1}-x_*^{i+1}\|^2}
		\le (1-\gamma\mu)\ec{\|x_t^i-x_*^i\|^2}+ 2 \gamma^2 \sigmass.
	\]
	Since $x_{t+1}-x_*=x_t^n-x_*^n$ and $x_t-x_*=x_t^0-x_*^0$, we can unroll the recursion, obtaining the epoch level recursion
  \[
    \ecn{x_{t+1} - x_\ast} \leq \br{1 - \gamma \mu}^{n} \ecn{x_t - x_\ast} + 2 \gamma^2 \sigmass \br{ \sum_{i=0}^{n-1} \br{1 - \gamma \mu}^{i} }.
  \]
  Unrolling this recursion across $T$ epochs, we obtain
  \begin{equation}
    \label{eq:thm_str_cvx_main-proof-2}
    \ecn{x_{T} - x_\ast} \leq \br{1 - \gamma \mu}^{nT} \sqn{x_0 - x_\ast} + 2 \gamma^2 \sigmass \br{\sum_{i=0}^{n-1} \br{1 - \gamma \mu}^{i}} \Biggl( \sum_{j=0}^{T-1} \br{1 - \gamma \mu}^{nj} \Biggr).
  \end{equation}
 The product of the two sums in \eqref{eq:thm_str_cvx_main-proof-2} can be bounded by  reparameterizing  the summation as follows:
  \begin{align*}
    \br{ \sum_{j=0}^{T-1} \br{1 - \gamma \mu}^{nj} } \br{\sum_{i=0}^{n-1} \br{1 - \gamma \mu}^{i}} &= \sum_{j=0}^{T-1} \sum_{i=0}^{n-1} \br{1 - \gamma \mu}^{nj + i} \\
    &= \sum_{k=0}^{nT-1} \br{1 - \gamma \mu}^{k} \leq \sum_{k=0}^{\infty} \br{1 - \gamma \mu}^{k} = \frac{1}{\gamma \mu}.
  \end{align*}
  Plugging this bound back into \eqref{eq:thm_str_cvx_main-proof-2}, we finally obtain the bound
  \[
    \ecn{x_{T} - x_\ast} \leq \br{1 - \gamma \mu}^{nT} \sqn{x_0 - x_\ast} + 2\gamma\frac{\sigmass}{\mu}.
  \]
\end{proof}

\subsection{Proof of complexity}
In this subsection, we show how we get from Theorem~\ref{thm:all-sc-rr-conv} the complexity for strongly convex functions.
\begin{corollary}
  \label{corr:all-sc-rr-conv}
  Under the same conditions as those in Theorem~\ref{thm:all-sc-rr-conv}, we choose stepsize 
  \[ \gamma = \min \pbr{ \frac{1}{L}, \frac{2}{\mu n T} \log\br{ \frac{\norm{x_0 - x_\ast} \mu T \sqrt{n}}{\sqrt{\kappa} \sigmaesc} } }. \]
The final iterate $x_T$ then satisfies
  \[ \ecn{x_{T} - x_\ast} = \mathcal{\tilde{O}} \br{ \exp\left( - \frac{\mu n T}{L}\right) \sqn{x_0 - x_\ast} + \frac{\kappa \sigmaesc^2}{\mu^2 n T^2} }, \]
  where $\tilde{O}(\cdot)$ denotes ignoring absolute constants and polylogarithmic factors. Thus, in order to obtain error (in squared distance to the optimum) less than $\e$, we require that the total number of iterations $n T$ satisfies
  \[ n T = \tilde{\Omega}\br{ \kappa + \frac{\sqrt{\kappa n} \sigmaesc}{\mu \sqrt{\e}} }. \]
\end{corollary}
\begin{proof}
  Applying \Cref{thm:all-sc-rr-conv}, the final iterate generated by Algorithms~\ref{alg:rr} or~\ref{alg:so} after $T$ epochs satisfies
  \[ \ecn{x_{T} - x_\ast} \leq \br{1 - \gamma \mu}^{n T} \sqn{x_0 - x_\ast} + 2\gamma\frac{\sigmass}{\mu}. \]
  Using \Cref{prop:shuffling-variance-normal-variance-bound} to bound $\sigmass$, we get
  \begin{equation}
    \label{eq:corr-all-sc-proof-1}
    \ecn{x_T - x_\ast} \leq \br{1 - \gamma \mu}^{n T} \sqn{x_0 - x_\ast} + \gamma^2 \kappa n \sigmaesc^2.
  \end{equation}
  We now have two cases: 
  \begin{itemize}[leftmargin=0.15in,itemsep=0.01in]
    \item \textbf{Case 1}: If $\frac{1}{L} \leq \frac{2}{\mu n T} \log\br{\frac{\norm{x_0 - x_\ast} \mu T \sqrt{n}}{\sqrt{\kappa} \sigmaesc} }$, then using $\gamma = \frac{1}{L}$ in \eqref{eq:corr-all-sc-proof-1} we have
    \begin{align*}
      \ecn{x_T - x_\ast} &\leq \br{1 - \frac{\mu}{L}}^{n T} \sqn{x_0 - x_\ast} + \frac{\kappa n \sigmaesc^2}{L^2} \\
      &\leq \br{1 - \frac{\mu}{L}}^{n T} \sqn{x_0 - x_\ast} + \frac{4 \kappa \sigmaesc^2}{\mu^2 n T^2} \log^2 \br{ \frac{\norm{x_0 - x_\ast} \mu T \sqrt{n}}{\sqrt{\kappa} \sigmaesc} }.
    \end{align*}
    Using that $1 - x \leq \exp(-x)$ in the previous inequality, we get
    \begin{equation}
      \label{eq:corr-all-sc-proof-2}
      \ecn{x_T - x_\ast} = \ctO\br{ \exp\br{- \frac{\mu n T}{L}} \sqn{x_0 - x_\ast} + \frac{\kappa \sigmaesc^2}{\mu^2 n T^2} },
    \end{equation}
    where $\ctO(\cdot)$ denotes ignoring polylogarithmic factors and absolute (non-problem specific) constants.
    \item \textbf{Case 2}: If $\frac{2}{\mu n T} \log\br{\frac{\norm{x_0 - x_\ast} \mu T \sqrt{n}}{\sqrt{\kappa} \sigmaesc} } < \frac{1}{L}$, recall that by \Cref{thm:all-sc-rr-conv},
    \begin{equation}
      \label{eq:corr-all-sc-proof-3}
      \ecn{x_T - x_\ast} \leq \br{1 - \gamma \mu}^{n T} \sqn{x_0 - x_\ast} + \gamma^2 \kappa n \sigmaesc^2.
    \end{equation}
    Plugging in $\gamma = \frac{2}{\mu n T} \log\br{\frac{\norm{x_0 - x_\ast} \mu T \sqrt{n}}{\sqrt{\kappa} \sigmaesc} }$,  the first term in \eqref{eq:corr-all-sc-proof-3} satisfies
    \begin{align}
      \br{1 - \gamma \mu}^{n T} \sqn{x_0 - x_\ast} &\leq \exp\br{- \gamma \mu n T} \sqn{x_0 - x_\ast} \nonumber \\
      &= \exp\br{ - 2 \log \br{ \frac{\norm{x_0 - x_\ast} \mu T \sqrt{n}}{\sqrt{\kappa} \sigmaesc} } } \sqn{x_0 - x_\ast} \nonumber \\
      &= \frac{\kappa \sigmaesc^2}{\mu^2 n T^2}.
      \label{eq:corr-all-sc-proof-4}
    \end{align}
    Furthermore, the second term in \eqref{eq:corr-all-sc-proof-3} satisfies
    \begin{equation}
      \gamma^2 \kappa n \sigmaesc^2 = \frac{4 \kappa \sigmaesc^2}{\mu^2 n T^2} \log^2 \br{ \frac{\norm{x_0 - x_\ast} \mu T \sqrt{n}}{\sqrt{\kappa} \sigmaesc} }.
      \label{eq:corr-all-sc-proof-5}
    \end{equation}
Substituting~\eqref{eq:corr-all-sc-proof-4} and~\eqref{eq:corr-all-sc-proof-5} into~\eqref{eq:corr-all-sc-proof-3}, we get
    \begin{equation}
      \label{eq:corr-all-sc-proof-6}
      \ecn{x_T - x_\ast} = \ctO\br{\frac{\kappa \sigmaesc^2}{\mu^2 n T^2} }.
    \end{equation}
    This concludes the second case.
  \end{itemize}
  It remains to take the maximum of~\eqref{eq:corr-all-sc-proof-2} from the first case and~\eqref{eq:corr-all-sc-proof-6} from the second case.
\end{proof}

\subsection{Two lemmas for Theorems~\ref{thm:only-f-sc-rr-conv} and~\ref{thm:weakly-convex-f-rr-conv}}

In order to prove Theorems~\ref{thm:only-f-sc-rr-conv} and~\ref{thm:weakly-convex-f-rr-conv}, it will be useful to define the following quantity.

\begin{definition}
  \label{def:cVt}
  Let $x_t^0, x_t^1, \ldots, x_t^n$ be iterates generated by Algorithms~\ref{alg:rr} or~\ref{alg:so}. We define the forward per-epoch deviation over the $t$-th epoch $\cV_t$ as
  \begin{equation}
    \label{eq:cVt_def}
    \cV_t\eqdef \sum_{i=0}^{n-1} \norm{x_t^i - x_{t+1}}^2.
  \end{equation}
\end{definition}

We will now establish two lemmas. First, we will show that $\cV_t$ can be efficiently upper bounded using Bregman divergences and the variance at the optimum. Subsequently use this bound to establish the convergence of RR/SO.

\subsubsection{Bounding the forward per-epoch deviation}
\begin{lemma}
  \label{lem:convex_rr_deviation}
	Consider the iterates of Random Reshuffling (Algorithm~\ref{alg:rr}) or Shuffle-Once (Algorithm~\ref{alg:so}). If the functions $f_1, \ldots, f_n$ are convex and Assumption~\ref{asm:f-smoothness} is satisfied, then 
  \begin{equation}
    \label{eq:lma-cVt-bound}
		\ec{\cV_t}
		\le 4\gamma^2n^2 L\sum_{i=0}^{n-1}\ec{D_{f_{\pi_i}}(x_*, x_t^i)}
		+ \frac{1}{2}\gamma^2n^2\sigma_*^2,
  \end{equation}
  where $\cV_t$ is defined as in Definition~\ref{def:cVt}, and $\sigmaesc^2$ is the variance at the optimum given by $\sigmaesc^2 \eqdef \frac{1}{n} \sum_{i=1}^{n} \sqn{\nabla f_{i} (x_\ast)}$.
\end{lemma}
\begin{proof}
	For any  fixed $k\in\{0,\dotsc, n-1\}$, by definition of $x_t^k$ and $x_{t+1}$ we get the decomposition
	\[
		x_t^k - x_{t+1}
		=\gamma\sum_{i=k}^{n-1}\nabla f_{\pi_i}(x_t^i)
		= \gamma\sum_{i=k}^{n-1}(\nabla f_{\pi_i}(x_t^i) - \nabla f_{\pi_i}(x_*))+\gamma\sum_{i=k}^{n-1}\nabla f_{\pi_i}(x_*).
	\]
	Applying Young's inequality to the sums above yields
	\begin{eqnarray*}
		\|x_t^k - x_{t+1}\|^2
		&\overset{\eqref{eq:sqnorm-triangle-2}}{\le}& 2\gamma^2\sqn{\sum_{i=k}^{n-1}(\nabla f_{\pi_i}(x_t^i) - \nabla f_{\pi_i}(x_*))} + 2\gamma^2\sqn{\sum_{i=k}^{n-1}\nabla f_{\pi_i}(x_*)} \\
		&\overset{\eqref{eq:sqnorm-sum-bound}}{\le}& 2\gamma^2n\sum_{i=k}^{n-1}\sqn{\nabla f_{\pi_i}(x_t^i) - \nabla f_{\pi_i}(x_*)} + 2\gamma^2\sqn{\sum_{i=k}^{n-1}\nabla f_{\pi_i}(x_*)}\\
		&\overset{\eqref{eq:bregman_lower_smooth}}{\le}& 4\gamma^2Ln\sum_{i=k}^{n-1}D_{f_{\pi_i}}(x_*, x_t^i) + 2\gamma^2\sqn{\sum_{i=k}^{n-1}\nabla f_{\pi_i}(x_*)} \\
		&\le& 4\gamma^2Ln\sum_{i=0}^{n-1}D_{f_{\pi_i}}(x_*, x_t^i) + 2\gamma^2\sqn{\sum_{i=k}^{n-1}\nabla f_{\pi_i}(x_*)} .
	\end{eqnarray*}
  Summing up and taking expectations leads to
  \begin{equation}
    \label{eq:cVt-bound-proof-1}
    \sum_{k=0}^{n-1} \ecn{x_t^k - x_{t+1}} \leq 4\gamma^2Ln^2\sum_{i=0}^{n-1}\ec{D_{f_{\pi_i}}(x_*, x_t^i)} + 2\gamma^2 \sum_{k=0}^{n-1} \ecn{\sum_{i=k}^{n-1}\nabla f_{\pi_i}(x_*)} .
  \end{equation}
  We now bound the second term in the right-hand side of \eqref{eq:cVt-bound-proof-1}. First,  using \Cref{lem:sampling_wo_replacement}, we get
	\begin{align*}
		\ecn{\sum_{i=k}^{n-1}\nabla f_{\pi_i}(x_*)}
		&= (n-k)^2\ecn{\frac{1}{n-k}\sum_{i=k}^{n-1}\nabla f_{\pi_i}(x_*)} \\
		&= (n-k)^2\frac{k}{(n-k)(n-1)}\sigma_*^2 \\
		&=\frac{k(n-k)}{n-1}\sigma_*^2.
	\end{align*}
	Next, by summing this for $k$ from 0 to $n-1$, we obtain
	\[
		\sum_{k=0}^{n-1} \ec{\sqn{\sum_{i=k}^{n-1}\nabla f_{\pi_i}(x_*)}} 
		= \sum_{k=0}^{n-1}\frac{k(n-k)}{n-1} \sigmaesc^2
		= \frac{1}{6}n(n+1) \sigmaesc^2 \le \frac{n^2 \sigmaesc^2}{4},
	\]
	where in the last step we also used $n\ge 2$. The result follows.
\end{proof}

\subsubsection{Finding a per-epoch recursion}
\begin{lemma}\label{lem:improved_convex}
	Assume that functions $f_1,\dotsc,f_n$ are  convex and that Assumption~\ref{asm:f-smoothness} is satisfied. If Random Reshuffling (Algorithm~\ref{alg:rr}) or Shuffle-Once (Algorithm~\ref{alg:so}) is run with a stepsize satisfying $\gamma\le \frac{1}{\sqrt{2}Ln}$, then
	\[
		\ec{\|x_{t+1}-x_*\|^2}
		\le \ec{\|x_t-x_*\|^2} - 2\gamma n\ec{f(x_{t+1})-f_*} + \frac{\gamma^3 L n^2\sigma_*^2}{2}.
	\]
\end{lemma}
\begin{proof}
  Define the sum of gradients used in the $t$-th epoch as $g_{t} \eqdef \sum_{i=0}^{n-1} \nabla f_{\pi_{i}} (x_t^i)$. We will use $g_t$ to relate the iterates $x_t$ and $x_{t+1}$. By definition of $x_{t+1}$, we can write
  \[ x_{t+1} = x_t^{n} = x_t^{n-1} - \gamma \nabla f_{\pi_{n-1}} (x_t^{n-1}) = \cdots = x_t^{0} - \gamma \sum_{i=0}^{n-1} \nabla f_{\pi_{i}} (x_t^i).  \]
Further, since $x_t^0 = x_t$, we see that $x_{t+1} = x_t - \gamma g_t$, which leads to
	\begin{align*}
		\|x_{t}-x_*\|^2
		=\|x_{t+1}+\gamma g_t-x_*\|^2
		&=\|x_{t+1}-x_*\|^2+2\gamma\ev{g_t, x_{t+1}-x_*}+\gamma^2\|g_t\|^2 \\
		&\ge \|x_{t+1}-x_*\|^2+2\gamma\ev{g_t, x_{t+1}-x_*} \\
		&= \|x_{t+1}-x_*\|^2+2\gamma\sum_{i=0}^{n-1}\ev{\nabla f_{\pi_i}(x_t^i), x_{t+1}-x_*}.
	\end{align*}
	Observe that for any $i$, we have the following decomposition
	\begin{align}
		\ev{\nabla f_{\pi_i}(x_t^i), x_{t+1}-x_*}
		&= [f_{\pi_i}(x_{t+1})-f_{\pi_i}(x_*)]+ [f_{\pi_i}(x_*) - f_{\pi_i}(x_t^i) - \<\nabla f_{\pi_i}(x_t^i), x_{t}^i-x_*>] \notag\\
		& \quad - [f_{\pi_i}(x_{t+1}) - f_{\pi_i}(x_t^i)-\<\nabla f_{\pi_i}(x_t^i), x_{t+1}-x_t^i>] \notag\\
		&= [f_{\pi_i}(x_{t+1})-f_{\pi_i}(x_*)] + D_{f_{\pi_i}}(x_*, x_t^i) - D_{f_{\pi_i}}(x_{t+1}, x_t^i). \label{eq:rr_decomposition}
	\end{align}
	Summing the first quantity in~\eqref{eq:rr_decomposition} over $i$ from $0$ to $n-1$ gives
	\[
		\sum_{i=0}^{n-1}[f_{\pi_i}(x_{t+1})-f_{\pi_i}(x_*)]
		= n(f(x_{t+1})-f_*).
	\]
	Now, we can bound the third term in the decomposition~\eqref{eq:rr_decomposition} using $L$-smoothness as follows:
	\[
		D_{f_{\pi_i}}(x_{t+1}, x_t^i)
		\le \frac{L}{2}\|x_{t+1}-x_t^i\|^2.
	\]
	By summing the right-hand side over $i$ from $0$ to $n-1$ we get the forward deviation over an epoch $\cV_t$, which we bound by Lemma~\ref{lem:convex_rr_deviation} to get
	\[
		\sum_{i=0}^{n-1}\ec{D_{f_{\pi_i}}(x_{t+1}, x_t^i)}
    \overset{\eqref{eq:cVt_def}}{\leq} \frac{L}{2} \ec{\cV_t} 
		\overset{\eqref{eq:lma-cVt-bound}}{\le}  2\gamma^2 L^2n^2\sum_{i=0}^{n-1}\ec{D_{f_{\pi_i}}(x_*, x_t^i)}
		+ \frac{\gamma^2 L n^2\sigma_*^2}{4}.
	\]
	Therefore, we can lower-bound the sum of the second and the third term in~\eqref{eq:rr_decomposition} as
	\begin{align*}
		\sum_{i=0}^{n-1}\ec{D_{f_{\pi_i}}(x_*, x_t^i) - D_{f_{\pi_i}}(x_{t+1}, x_t^i)}
		&\ge \sum_{i=0}^{n-1}\ec{D_{f_{\pi_i}}(x_*, x_t^i)}  - 2\gamma^2L^2n^2\sum_{i=0}^{n-1}\ec{D_{f_{\pi_i}}(x_*, x_t^i)}\\
		&\qquad - \frac{\gamma^2 L n^2\sigma_*^2}{4} \\
		&\ge (1-2\gamma^2L^2n^2)\sum_{i=0}^{n-1}\ec{D_{f_{\pi_i}}(x_*, x_t^i)} - \frac{\gamma^2 L n^2\sigma_*^2}{4} \\
		&\ge  - \frac{\gamma^2 L n^2\sigma_*^2}{4},
	\end{align*}
	where in the third inequality we used that $\gamma \leq \frac{1}{\sqrt{2} L n}$ and that $D_{f_{\pi_{i}}} (x_\ast, x_t^i)$ is nonnegative. Plugging this back into the lower-bound on $\|x_t-x_*\|^2$ yields
	\[
		\ecn{x_t - x_\ast}
		\ge \ec{\|x_{t+1}-x_*\|^2} + 2\gamma n\ec{f(x_{t+1})-f_*}- \frac{\gamma^3 L n^2\sigma_*^2}{2}.
	\]
	Rearranging the terms gives the result.
\end{proof}

\subsection{Proof of Theorem~\ref{thm:only-f-sc-rr-conv}}
\begin{proof}
	We can use \Cref{lem:improved_convex} and strong convexity to obtain
	\begin{align*}
		\ecn{x_{t+1}-x_*}
		&\le \ecn{x_t-x_*} - 2\gamma n\ec{f(x_{t+1})-f_*} + \frac{\gamma^3 L n^2\sigma_*^2}{2}\\
		&\overset{\eqref{eq:mu-convexity}}{\le}  \ecn{x_t-x_*} - \gamma n\mu\ecn{x_{t+1}-x_*} + \frac{\gamma^3 L n^2\sigma_*^2}{2},
	\end{align*}
	whence
	\begin{align*}
		\ecn{x_{t+1}-x_*}
    &\le\frac{1}{1+\gamma\mu n}\left( \ecn{x_t-x_*} + \frac{\gamma^3 L n^2\sigma_*^2}{2} \right) \\
    &= \frac{1}{1 + \gamma \mu n} \ecn{x_t - x_\ast} +  \frac{1}{1 + \gamma \mu n} \frac{\gamma^3 L n^2\sigma_*^2}{2} \\
		&\le \left(1 - \frac{\gamma\mu n}{2}\right) \ecn{x_t - x_\ast} + \frac{\gamma^3 L n^2 \sigmaesc^2}{2} .
  \end{align*}  
  Recursing for $T$ iterations, we get that the final iterate satisfies
  \begin{align*}
    \ecn{x_T - x_\ast} &\leq \br{ 1 - \frac{\gamma \mu n}{2} }^{T} \sqn{x_0 - x_\ast} + \frac{\gamma^3 L n^2 \sigmaesc^2 }{2} \br{ \sum_{j=0}^{T-1} \br{1 - \frac{\gamma \mu n}{2}}^{j} } \\
    &\leq \br{ 1 - \frac{\gamma \mu n}{2} }^{T} \sqn{x_0 - x_\ast} + \frac{\gamma^3 L n^2 \sigmaesc^2 }{2} \br{ \frac{2}{\gamma \mu n} } \\
    &= \br{ 1 - \frac{\gamma \mu n}{2} }^T \sqn{x_0 - x_\ast} + \gamma^2 \kappa n \sigmaesc^2.
    \qedhere
  \end{align*}
\end{proof}

\subsection{Proof of Theorem~\ref{thm:weakly-convex-f-rr-conv}}
\begin{proof}
  We start with \Cref{lem:improved_convex}, which states that the following inequality holds:
  \[ \ec{\sqn{x_{t+1} - x_\ast}} \leq \ecn{x_t - x_\ast} - 2 \gamma n \ec{f(x_{t+1}) - f (x_\ast)} + \frac{\gamma^3 L n^2 \sigmaesc^2}{2} . \]
  Rearranging the result leads to
  \[ 2 \gamma n \ec{f(x_{t+1}) - f(x_\ast)} \leq \ecn{x_{t} - x_\ast} - \ecn{x_{t+1} - x_\ast} +\frac{\gamma^3 L n^2 \sigmaesc^2}{2} .  \]
  Summing these inequalities for $t=0,1,\dots, T-1$ gives
  \begin{align*}
    2 \gamma n \sum_{t=0}^{T-1} \ec{f(x_{t+1}) - f(x_\ast)} &\leq \sum_{t=0}^{T-1} \br{ \ecn{x_t - x_\ast} - \ecn{x_{t+1} - x_\ast} } + \frac{\gamma^3 L n^2 \sigmaesc^2 T}{2}  \\
    &= \sqn{x_{0} - x_\ast} - \ecn{x_{T} - x_\ast} +\frac{\gamma^3 L n^2 \sigmaesc^2 T}{2} \\
    &\leq \sqn{x_0 - x_\ast} + \frac{\gamma^3 L n^2 \sigmaesc^2 T}{2} ,
  \end{align*}
  and dividing both sides by $2 \gamma n T$, we get
  \[
    \frac{1}{T} \sum_{t=0}^{T-1} \ec{f(x_{t+1}) - f(x_\ast)} \leq \frac{\sqn{x_0 - x_\ast}}{2 \gamma n T} + \frac{ \gamma^2 L  n \sigmaesc^2}{4}. 
  \]
  Finally, using convexity of $f$,  the average iterate $\hat{x}_{T} \eqdef \frac{1}{T} \sum_{t=1}^{T} x_t$ satisfies
	\[
		\ec{f(\hat{x}_T) - f(x_\ast)} \leq \frac{1}{T} \sum_{t=1}^{T} \ec{f(x_t)-f(x_\ast)}
    \le \frac{\sqn{x_0 - x_\ast}}{2 \gamma n T} + \frac{ \gamma^2 L  n \sigmaesc^2}{4}. 
    \qedhere
	\]
\end{proof}

\subsection{Proof of complexity}
\begin{corollary}
  \label{corr:weakly-convex-f}
  Under the same conditions as Theorem~\ref{thm:weakly-convex-f-rr-conv}, choose the stepsize 
  \[ \gamma = \min \pbr{ \frac{1}{\sqrt{2} L n}, \br{ \frac{\sqn{x_0 - x_\ast}}{L n^2 T \sigmaesc^2} }^{1/3} }. \]
  Then
  \[ \ec{f(\hat{x}_T) - f(x_\ast)} \leq \frac{L \sqn{x_0 - x_\ast}}{\sqrt{2} T} + \frac{3 L^{1/3} \norm{x_0 - x_\ast}^{4/3} \sigmaesc^{2/3}}{4 n^{1/3} T^{2/3}}. \]
We can guarantee $\ec{f(\hat{x}_T) - f(x_\ast)} \leq \e^2$ provided that the total number of iterations satisfies
  \[ Tn \geq  \frac{2\sqn{x_0 - x_\ast} \sqrt{Ln}}{\e^2} \max \pbr{ \sqrt{2 L n}, \frac{\sigmaesc}{\e} }. \]
\end{corollary}
\begin{proof}
  We start with the guarantee of Theorem~\ref{thm:weakly-convex-f-rr-conv}:
  \begin{equation}
    \label{eq:corr-wc-proof-1}
    \ec{f(\hat{x}_T) - f(x_\ast)} \leq \frac{\sqn{x_0 - x_\ast}}{2 \gamma n T} + \frac{\gamma^2 L n \sigmaesc^2}{4} .
  \end{equation}
  We now have two cases depending on the stepsize: 
  \begin{itemize}[leftmargin=0.15in,itemsep=0.01in]
    \item \textbf{Case 1}: If $\gamma = \frac{1}{\sqrt{2} L n} \leq \br{ \frac{\sqn{x_0 - x_\ast}}{L n^2 T \sigmaesc^2} }^{1/3}$, then plugging this $\gamma$ into \eqref{eq:corr-wc-proof-1} gives
    \begin{align}
      \ec{f(\hat{x}_T) - f(x_\ast)} &\leq \frac{L \sqn{x_0 - x_\ast}}{\sqrt{2} T} + \frac{\gamma^2 L n \sigmaesc^2 }{4}  \nonumber \\
      &\leq \frac{L \sqn{x_0 - x_\ast}}{\sqrt{2} T} +  \br{ \frac{\sqn{x_0 - x_\ast}}{L n^2 T \sigmaesc^2} }^{2/3} \frac{Ln \sigmaesc^2}{4}  \nonumber \\
      &= \frac{L \sqn{x_0 - x_\ast}}{\sqrt{2} T} + \frac{L^{1/3} \sigmaesc^{2/3} \norm{x_0 - x_\ast}^{4/3}}{4 n^{1/3} T^{2/3}}.
      \label{eq:corr-wc-proof-2}
    \end{align}
    \item \textbf{Case 2}: If $\gamma = \br{ \frac{\sqn{x_0 - x_\ast}}{L n^2 T \sigmaesc^2} }^{1/3} \leq \frac{1}{\sqrt{2} L n}$, then plugging this $\gamma$ into \eqref{eq:corr-wc-proof-1} gives
    \begin{align}
      \ec{f(\hat{x}_T) - f(x_\ast)} &\leq  \frac{L^{1/3} \norm{x_0 - x_\ast}^{4/3} \sigmaesc^{2/3}}{2 n^{1/3} T^{2/3}} + \frac{L^{1/3} \sigmaesc^{2/3} \norm{x_0 - x_\ast}^{4/3}}{4 n^{1/3} T^{2/3}} \nonumber \\
      &= \frac{3 L^{1/3} \norm{x_0 - x_\ast}^{4/3} \sigmaesc^{2/3}}{4 n^{1/3} T^{2/3}}.
      \label{eq:corr-wc-proof-3}
    \end{align}
  \end{itemize}
  Combining \eqref{eq:corr-wc-proof-2} and \eqref{eq:corr-wc-proof-3}, we see that in both cases we have
  \[ \ec{f(\hat{x}_T) - f(x_\ast)} \leq \frac{L \sqn{x_0 - x_\ast}}{\sqrt{2} T} + \frac{3 L^{1/3} \norm{x_0 - x_\ast}^{4/3} \sigmaesc^{2/3}}{4 n^{1/3} T^{2/3}}. \]
  Translating this to sample complexity, we can guarantee that $\ec{f(\hat{x}_T) - f(x_\ast)} \leq \e^2$ provided
  \[
    n T \geq \frac{2 \sqn{x_0 - x_\ast} \sqrt{L n}}{\e^2} \max \pbr{ \sqrt{L n}, \frac{\sigmaesc}{\e} }. 
    \qedhere
  \]
\end{proof}

\clearpage
\section{Proofs for non-convex objectives (Section~\ref{sec:non-convex})}

\subsection{Proof of Proposition~\ref{prop:2nd-moment-bound}}
\begin{proof}
  This proposition is a special case of Lemma~3 in \citep{Khaled2020} and we prove it here for completeness. Let $x \in \R^d$. We start with \eqref{eq:grad-bound} (which does not require convexity) applied to each $f_i$:
  \[ \sqn{\nabla f_{i} (x)} \leq 2 L \br{f_{i} (x) - f_i^\ast}. \]
  Averaging, we derive
  \begin{align*}
    \frac{1}{n} \sum_{i=1}^{n} \sqn{\nabla f_{i} (x)} &\leq 2 L \br{f(x) - \frac{1}{n} \sum_{i=1}^{n} f_i^\ast} \\
    &= 2 L \br{f(x) - f_\ast} + 2 L \br{f_\ast - \frac{1}{n} \sum_{i=1}^{n} f_i^\ast}.
  \end{align*}
  Note that because $f_\ast$ is the infimum of $f(\cdot)$ and $\frac{1}{n} \sum_{i=1}^{n} f_i^\ast$ is a lower bound on $f$ then $f_\ast - \frac{1}{n} \sum_{i=1}^{n} f_i^\ast \geq 0$. We may now use the variance decomposition
  \begin{eqnarray*}
    \frac{1}{n} \sum_{i=1}^{n} \sqn{\nabla f_{i} (x) - \nabla f(x)} &\overset{\eqref{eq:variance-decomp}}{=}& \frac{1}{n} \sum_{i=1}^{n} \sqn{\nabla f_i (x)} - \sqn{\nabla f(x)} \\
    &\leq& \frac{1}{n} \sum_{i=1}^{n} \sqn{\nabla f_{i} (x)} \\
    &\leq& 2 L \br{f(x) - f_\ast} + 2 L \br{f_\ast - \frac{1}{n} \sum_{i=1}^{n} f_i^\ast}.
  \end{eqnarray*}
  It follows that Assumption~\ref{asm:2nd-moment} holds with $A = L$ and $B^2 = 2 L \br{f_\ast - \frac{1}{n} \sum_{i=1}^{n} f_i^\ast}$.
\end{proof}

\subsection{Finding a per-epoch recursion}
For this subsection and the rest of this section, we need to define the following quantity:
\begin{definition}
  \label{def:vt}
  For Algorithm~\ref{alg:rr} we define the \emph{backward per-epoch deviation} at timestep $t$ by
  \[ V_{t} \eqdef \frac{1}{n} \sum_{i=1}^{n} \sqn{x_t^i - x_t}. \]
\end{definition}
We will study the convergence of Algorithm~\ref{alg:rr} for non-convex objectives as follows: we first derive a per-epoch recursion that involves $V_t$ in Lemma~\ref{lemma:epoch-recursion-non-convex}, then we show that $V_t$ can be bounded using smoothness and probability theory in Lemma~\ref{lemma:vt_rr}, and finally combine these two to prove Theorem~\ref{thm:rr-nonconvex}.

\begin{lemma}
  \label{lemma:epoch-recursion-non-convex}
  Suppose that Assumption~\ref{asm:f-smoothness} holds. Then for iterates $x_t$ generated by Algorithm~\ref{alg:rr} with stepsize $\gamma \leq \frac{1}{L n}$, we have
  \begin{equation}
    \label{eq:epoch-recursion-non-convex}
    f(x_{t+1}) \leq f(x_t) - \frac{\gamma n}{2} \sqn{\nabla f(x_t)} + \frac{\gamma L^2}{2} V_t,
  \end{equation}
  where $V_t$ is defined as in \Cref{def:vt}.
\end{lemma}
\begin{proof}
Our approach for establishing this lemma is similar to that of \citep[Theorem~1]{Nguyen2020}, which we became aware of in the course of preparing this manuscript. Recall that $x_{t+1} = x_t - \gamma g_t$, where $g_t = \sum_{i=0}^{n-1} \nabla f_{\pi_{i}} (x_t^i)$. Using $L$-smoothness of $f$, we get
  \begin{eqnarray}
    f(x_{t+1}) &\overset{\eqref{eq:L-smoothness}}{\leq}& f(x_{t}) + \ev{\nabla f(x_t), x_{t+1} - x_t} + \frac{L}{2} \sqn{x_{t+1} - x_t} \nonumber \\
    &=& f(x_t) - \gamma n \ev{\nabla f(x_t), \frac{g_t}{n}} + \frac{ \gamma^2 L n^2}{2} \sqn{ \frac{g_t}{n} } \nonumber \\
    &\overset{\eqref{eq:square-decompos}}{=}& f(x_t) - \frac{\gamma n}{2} \br{ \sqn{\nabla f(x_t)} + \sqn{\frac{g_t}{n}} - \sqn{ \nabla f(x_t) - \frac{g_t}{n} } } + \frac{\gamma^2 L  n^2}{2} \sqn{ \frac{g_t}{n} } \nonumber \\
    &=& f(x_t) - \frac{\gamma n}{2} \sqn{\nabla f(x_t)} - \frac{\gamma n}{2} \br{1 - L \gamma n} \sqn{\frac{g_t}{n}} + \frac{\gamma n}{2} \sqn{\nabla f(x_t) - \frac{g_t}{n}}.
    \label{eq:ernc-1}
  \end{eqnarray}
  By assumption, we have $\gamma \leq \frac{1}{Ln}$, and hence $1 - L \gamma n \geq 0$. Using this in \eqref{eq:ernc-1}, we get
  \begin{equation}
    \label{eq:ernc-2}
    f(x_{t+1}) \leq f(x_t) - \frac{\gamma n}{2} \sqn{\nabla f(x_t)} + \frac{\gamma n}{2} \sqn{\nabla f(x_t) - \frac{g_t}{n}}. 
  \end{equation}
  For the last term in \eqref{eq:ernc-2}, we note
  \begin{eqnarray}
    \sqn{\nabla f(x_t) - \frac{g_t}{n}} &=& \sqn{ \frac{1}{n} \sum_{i=0}^{n-1} \left [ \nabla f_{\prm{i}} (x_t) - \nabla f_{\prm{i}} (x_t^i) \right ] } \nonumber \\
    &\overset{\eqref{eq:sqnorm-jensen}}{\leq}&  \frac{1}{n} \sum_{i=0}^{n-1} \sqn{\nabla f_{\prm{i}} (x_t) - \nabla f_{\prm{i}} (x_t^i) } \nonumber \\
    &\overset{\eqref{eq:nabla-Lip}}{\leq}&  \frac{1}{n} \sum_{i=0}^{n-1} L^2 \sqn{x_t - x_t^i} = \frac{L^2}{n} V_t.
    \label{eq:ernc-3}
  \end{eqnarray}
  Plugging in~\eqref{eq:ernc-3} into~\eqref{eq:ernc-2} yields the lemma's claim.
\end{proof}

\subsection{Bounding the backward per-epoch deviation}
\begin{lemma}
  \label{lemma:vt_rr}
  Suppose that Assumption~\ref{asm:f-smoothness} holds (with each $f_{i}$ possibly non-convex) and that Algorithm~\ref{alg:rr} is used with a stepsize $\gamma \le \frac{1}{2 L n}$. Then
  \begin{equation}
    \label{eq:vt_bound_rr}
		\et{V_t}
		\le \gamma^2n^3\sqn{\nabla f(x_t)} + \gamma^2n^2\sigma_{t}^2,
	\end{equation}
  where $V_t$ is defined as in \Cref{def:vt} and $\sigma_t^2 \eqdef \frac{1}{n} \sum_{j=1}^{n} \sqn{\nabla f_{j} (x_t) - \nabla f(x_t)}$. 
\end{lemma}
\begin{proof}
  Let us fix any $k\in[1, n-1]$ and find an upper bound for $\et{\norm{x_t^k - x_t}^2}$. First, note that
	\[
		x_t^k = x_t - \gamma\sum_{i=0}^{k-1}\nabla f_{\pi_i}(x_t^i).
	\]
	Therefore, by Young's inequality, Jensen's inequality and gradient Lipschitzness
	\begin{eqnarray*}
		\et{\|x_t^k - x_t\|^2}
		&=& \gamma^2\et{\left\lVert\sum_{i=0}^{k-1}\nabla f_{\pi_i}(x_t^i) \right\rVert^2} \\
		&\overset{\eqref{eq:sqnorm-triangle-2}}{\le}& 2\gamma^2\et{\left\lVert\sum_{i=0}^{k-1}\left(\nabla f_{\pi_i}(x_t^i)-\nabla f_{\pi_i}(x_t)\right)  \right\rVert^2} + 2\gamma^2\et{\left\lVert\sum_{i=0}^{k-1}\nabla f_{\pi_i}(x_t) \right\rVert^2} \\
		&\overset{\eqref{eq:sqnorm-sum-bound}}{\le}& 2\gamma^2k\sum_{i=0}^{k-1}\et{\norm{\nabla f_{\pi_i}(x_t^i)-\nabla f_{\pi_i}(x_t)}^2} + 2\gamma^2\et{\left\lVert\sum_{i=0}^{k-1}\nabla f_{\pi_i}(x_t) \right\rVert^2} \\
		&\overset{\eqref{eq:nabla-Lip}}{\le}& 2\gamma^2 L^2 k\sum_{i=0}^{k-1}\et{\|x_t^i-x_t\|^2} + 2\gamma^2\et{\left\lVert\sum_{i=0}^{k-1}\nabla f_{\pi_i}(x_t) \right\rVert^2}.
	\end{eqnarray*}
	Let us bound the second term. For any $i$ we have $\et{\nabla f_{\pi_i}(x_t)}=\nabla f(x_t)$, so using \Cref{lem:sampling_wo_replacement} (with vectors $\nabla f_{\pi_0} (x_t), \nabla f_{\pi_{1}} (x_t), \ldots, \nabla f_{\pi_{k-1}} (x_t)$) we obtain
	\begin{eqnarray*}
		\et{\left\lVert\sum_{i=0}^{k-1}\nabla f_{\pi_i}(x_t) \right\rVert^2}
		&\overset{\eqref{eq:rv_moments}}{=}& k^2\sqn{\nabla f(x_t)} + k^2\et{\left\lVert\frac{1}{k}\sum_{i=0}^{k-1}  (\nabla f_{\pi_i}(x_t) - \nabla f(x_t))\right\rVert^2} \\
		&\overset{\eqref{eq:sampling_wo_replacement}}{\le}& k^2\sqn{\nabla f(x_t)}+ \frac{k(n-k)}{n-1}\sigma_t^2.
	\end{eqnarray*}
	where $\sigma_t^2 \eqdef \frac{1}{n} \sum_{j=1}^{n} \sqn{\nabla f_{j} (x_t) - \nabla f(x_t)}$. Combining the produced bounds yields 
	\begin{align*}
		\et{\norm{x_t^k - x_t}^2}
		&\le 2\gamma^2 L^2 k \sum_{i=0}^{k-1}\et{\norm{x_t^i-x_t}^2} + 2\gamma^2 k^2\sqn{\nabla f(x_t)} + 2\gamma^2\frac{k(n-k)}{n-1}\sigma_t^2 \\
		&\le 2\gamma^2 L^2 k \ec{V_t} + 2\gamma^2 k^2\sqn{\nabla f(x_t)} + 2\gamma^2\frac{k(n-k)}{n-1}\sigma_t^2,
	\end{align*}
	whence
	\begin{align*}
		\ec{V_t}
		& = \sum_{k=0}^{n-1} \et{\|x_t^k - x_t\|^2} \\
    & \le \gamma^2 L^2 n(n-1) \ec{V_t} + \frac{1}{3}\gamma^2(n-1)n(2n-1)\sqn{\nabla f(x_t)} 
    + \frac{1}{3}\gamma^2 n(n+1)\sigma_t^2.
  \end{align*}
  Since $\ec{V_t}$ appears in both sides of the equation, we rearrange and use that $\gamma\le \frac{1}{2Ln}$ by assumption, which leads to
	\begin{align*}
		\ec{V_t} 
		&\le  \frac{4}{3}(1 - \gamma^2L^2n(n-1)) \ec{V_t} \\
		&\le  \frac{4}{9}\gamma^2(n-1)n(2n-1)\sqn{\nabla f(x_t)} + \frac{4}{9}\gamma^2 n(n+1)\sigma_t^2 \\
    &\le  \gamma^2n^3\sqn{\nabla f(x_t)} + \gamma^2n^2\sigma_t^2.
    \qedhere
	\end{align*}
\end{proof}

\subsection{A lemma for solving the non-convex recursion}
\begin{lemma}
  \label{lemma:noncvx-recursion-solution}
  Suppose that there exist constants $a, b, c \geq 0$ and nonnegative sequences $(s_{t})_{t=0}^{T}, (q_{t})_{t=0}^{T}$ such that for any $t$ satisfying $0 \leq t \leq T$ we have the recursion
  \begin{equation}
    \label{eq:nonconvex-recursion-init}
    s_{t+1} \leq \br{1 + a} s_{t} - b q_{t} + c.
  \end{equation}
  Then, the following holds:
  \begin{equation}
    \label{eq:nonconvex-recursion-soln}
    \min_{t=0, \ldots, T-1} q_{t} \leq \frac{\br{1+a}^T}{b T} s_{0} + \frac{c}{b}.
  \end{equation}
\end{lemma}
\begin{proof}
  The first part of the proof (for $a > 0$) is a distillation of the recursion solution in Lemma~2 of \citet{Khaled2020} and we closely follow their proof. Define 
  \[ w_{t} \eqdef \frac{1}{\br{1+a}^{t+1}}. \]
  Note that $w_{t} \br{1+a} = w_{t-1}$ for all $t$. Multiplying both sides of \eqref{eq:nonconvex-recursion-init} by $w_{t}$,
  \[
    w_{t} s_{t+1} \leq \br{1 + a} w_t s_t - b w_t q_t + c w_t = w_{t-1} s_{t} - b w_{t} q_t + c w_t.
  \]
  Rearranging, we get
$
    b w_{t} q_{t} \leq w_{t-1} s_{t} - w_{t} s_{t+1} + c w_{t}.
$
  Summing up as $t$ varies from $0$ to $T-1$ and noting that the sum telescopes leads to
  \begin{align*}
    \sum_{t=0}^{T-1} b w_{t} q_{t} &\leq \sum_{t=0}^{T-1} \br{w_{t-1} s_{t} - w_{t} s_{t+1}} + c \sum_{t=0}^{T-1} w_{t} \\
    &= w_{0} s_{0} - w_{T-1} s_{T} + c \sum_{t=0}^{T-1} w_{t} \\
    &\leq w_{0} s_{0} + c \sum_{t=0}^{T-1} w_{t}.
  \end{align*}
  Let $W_{T} = \sum_{t=0}^{T-1} w_{t}$. Dividing both sides by $W_{T}$, we get
  \begin{align}
    \frac{1}{W_{T}} \sum_{t=0}^{T-1} b w_{t} q_{t} \leq \frac{w_0 s_0}{W_T} + c.
    \label{eq:nc-rec-1}
  \end{align}
  Note that the left-hand side of \eqref{eq:nc-rec-1} satisfies
  \begin{equation}
    b \min_{t=0, \ldots, T-1} q_t \leq \frac{1}{W_T} \sum_{t=0}^{T-1} b w_t q_t.
    \label{eq:nc-rec-2}
  \end{equation}
 For the right-hand side of \eqref{eq:nc-rec-1}, we have
  \begin{equation}
    W_{T} = \sum_{t=0}^{T-1} w_{t} \geq T \min_{t=0, \ldots, T-1} w_{t} = T w_{T-1} = \frac{T}{\br{1+a}^{T}}.
    \label{eq:nc-rec-3}
  \end{equation}
  Substituting with \eqref{eq:nc-rec-3} in \eqref{eq:nc-rec-2} and dividing both sides by $b$, we finally get
  \[
    \min_{t=0, \ldots, T-1} q_{t} \leq \frac{\br{1 + a}^T}{bT} s_0 + \frac{c}{b}.
    \qedhere
  \]
\end{proof}

\subsection{Proof of Theorem~\ref{thm:rr-nonconvex}}
\begin{proof}
  \textbf{Without PL.} Taking expectation in \Cref{lemma:epoch-recursion-non-convex} and then using \Cref{lemma:vt_rr}, we have that for any $t \in \{ 0, 1, \ldots, T-1 \}$,
  \begin{eqnarray*}
    \et{f(x_{t+1})} &\overset{\eqref{eq:epoch-recursion-non-convex}}{\leq} & f(x_t) - \frac{\gamma n}{2} \sqn{\nabla f(x_t)} + \frac{\gamma L^2}{2} \et{V_t} \\
    &\overset{\eqref{eq:vt_bound_rr}}{\leq} & f(x_t) - \frac{\gamma n}{2} \sqn{\nabla f(x_t)} + \frac{\gamma L^2}{2} \br{\gamma^2 n^3 \sqn{\nabla f(x_t)} + \gamma^2 n^2 \sigma_{t}^2} \\
    &= & f(x_t) - \frac{\gamma n}{2} \br{1 - \gamma^2 L^2 n^2} \sqn{\nabla f(x_t)} + \frac{ \gamma^3 L^2 n^2 \sigma_{t}^2}{2} .
  \end{eqnarray*}
  Let $\delta_{t} \eqdef f(x_t) - f_\ast$. Adding $-f_\ast$ to both sides and using Assumption~\ref{asm:2nd-moment},
  \begin{align*}
    \et{\delta_{t+1}}  & \leq  \delta_{t} - \frac{\gamma n}{2} \br{1 - \gamma^2 L^2 n^2} \sqn{\nabla f(x_t)} + \frac{\gamma^3 L^2  n^2 \sigma_{t}^2}{2}  \\
    & \leq  \br{1 + \gamma^3  A L^2 n^2} \delta_{t} - \frac{\gamma n}{2} \br{1 - \gamma^2 L^2 n^2} \sqn{\nabla f(x_t)} + \frac{ \gamma^3 L^2 n^2 B^2}{2} .
  \end{align*}
  Taking unconditional expectations in the last inequality and using that by assumption on $\gamma$ we have $1 - \gamma^2 L^2 n^2 \geq \frac{1}{2}$, we get the estimate
  \begin{equation}
    \label{eq:thm-rr-nc-1}
    \ec{\delta_{t+1}} \leq \br{1 + \gamma^3 A L^2  n^2} \ec{\delta_{t}} - \frac{\gamma n}{4} \ecn{\nabla f(x_t)} + \frac{ \gamma^3 L^2 n^2 B^2}{2} .
  \end{equation}
  Comparing \eqref{eq:nonconvex-recursion-init} with \eqref{eq:thm-rr-nc-1} verifies that the conditions of \Cref{lemma:noncvx-recursion-solution} are readily satisfied. Applying the lemma, we get
  \[
    \min_{t=0, \ldots, T-1} \ecn{\nabla f(x_t)} \leq \frac{4 \br{1 + \gamma^3 A L^2 n^2}^{T}}{\gamma n T} \br{f(x_0) - f_\ast} + 2 \gamma^2 L^2  n B^2.
  \]
  Using that $1 + x \leq \exp(x)$ and that the stepsize $\gamma$ satisfies $\gamma \leq \br{A L^2 n^2 T}^{-1/3}$, we have
  \[ \br{1 + \gamma^3 A L^2  n^2}^{T} \leq \exp\br{\gamma^3 A L^2  n^2 T} \leq \exp\br{1} \leq 3. \]
  Using this in the previous bound, we finally obtain
  \[
    \min_{t=0, \ldots, T-1} \ecn{\nabla f(x_t)} \leq \frac{12 \br{f(x_0) - f_\ast}}{\gamma n T} + 2  \gamma^2 L^2 n B^2.
  \]
  
  \textbf{With PL.} Now we additionally assume that $A=0$ and that $\frac{1}{2}\|\nabla f(x)\|^2\ge \mu(f(x)-f_*)$. Then, \eqref{eq:thm-rr-nc-1} yields
  \begin{align*}
  	\ec{\delta_{t+1}}
  	&\leq \ec{\delta_{t}} - \frac{\gamma n}{4} \ecn{\nabla f(x_t)} + \frac{ \gamma^3 L^2 n^2 B^2}{2} \\
  	&\le \ec{\delta_{t}} - \frac{\gamma \mu n}{2} \ec{f(x_t)-f_*} + \frac{ \gamma^3 L^2 n^2 B^2}{2} \\
  	&= \br{1-\frac{\gamma\mu n}{2}}\ec{\delta_t} +  \frac{ \gamma^3 L^2 n^2 B^2}{2} .
  \end{align*}
  As in the proof of Theorem~\ref{thm:only-f-sc-rr-conv}, we recurse this bound to $x_0$:
  \begin{align*}
	  \ec{\delta_{T}}
	  &\le \br{1-\frac{\gamma\mu n}{2}}^T\delta_0 + \frac{ \gamma^3 L^2 n^2 B^2}{2} \sum_{j=0}^{T-1}\br{1-\frac{\gamma\mu n}{2}}^j \\
	  &\le \br{1-\frac{\gamma\mu n}{2}}^T\delta_0 + \frac{ \gamma^3 L^2 n^2 B^2}{2} \frac{2}{\gamma\mu n} \\
	  &= \br{1-\frac{\gamma\mu n}{2}}^T\delta_0 + \gamma^2\kappa L n B^2.
  \qedhere
  \end{align*}
\end{proof}

\subsection{Proof of complexity}
\begin{corollary}
  \label{corr:rr-nonconvex}
  Choose the stepsize $\gamma$ as
  \[ \gamma = \min \pbr{ \frac{1}{2 L n}, \frac{1}{A^{1/3} L^{2/3} n^{2/3} T^{1/3} }, \frac{\e}{2 L \sqrt{n} B} }. \]
  Then the minimum gradient norm satisfies $$\min \limits_{t=0, \ldots, T-1} \ecn{ \nabla f(x_t)} \leq \e^2$$ provided the total number of iterations satisfies
  \[ Tn \geq \frac{48 \delta_{0} L \sqrt{n}}{\e^2} \max \pbr{ \sqrt{n}, \frac{\sqrt{6 \delta_{0} A}}{\e}, \frac{B}{\e} }.  \]
\end{corollary}
\begin{proof}
  From \Cref{thm:rr-nonconvex}
  \[ \min_{t=0, \ldots, T-1} \ecn{\nabla f(x_t)} \leq \frac{12 \br{f(x_0) - f_\ast}}{\gamma n T} + 2 \gamma^2 L^2  n B^2.  \]
  Note that by condition on the stepsize $2 L^2 \gamma^2 n B^2 \leq \e^2/2$, hence 
  \[ \min_{t=0, \ldots, T-1} \ecn{\nabla f(x_t)} \leq \frac{12 \br{f(x_0) - f_\ast}}{\gamma n T} + \frac{\e^2}{2}. \]
  Thus, to make the squared gradient norm smaller than $\e^2$ we require
  \[ \frac{12 \br{f(x_0) - f_\ast}}{\gamma n T} \leq \frac{\e^2}{2}, \]
  or equivalently
  \begin{equation}
    \label{corr:rr-nonconvex-proof-1}
    n T \geq \frac{24 \br{f(x_0) - f_\ast}}{\e^2 \gamma} = \frac{24 \delta_{0}}{\e^2} \max \pbr{ 2 L n, \br{A L^2 n^2 T}^{1/3}, \frac{2 L \sqrt{n} B}{\e} }, 
  \end{equation}
  where $\delta_{0} \eqdef f(x_0) - f_\ast$ and where we plugged in the value of the stepsize $\gamma$ we use. Note that $nT$ appears on both sides in the second term in the maximum in \eqref{corr:rr-nonconvex-proof-1}, hence we can cancel out and simplify:
  \[ n T \geq \frac{24 \delta_{0}}{\e^2} (A L^2 n^2 T)^{1/3} \Longleftrightarrow n T \geq \frac{(24 \delta_{0})^{3/2} L \sqrt{An}}{\e^3}.   \]
  Using this simplified bound in \eqref{corr:rr-nonconvex-proof-1} we obtain that $\min_{t=0, \ldots, T-1} \ecn{\nabla f(x_t)} \leq \e^2$ provided
  \[
    n T \geq \frac{48 \delta_{0} L \sqrt{n}}{\e^2} \max \pbr{ \sqrt{n}, \frac{\sqrt{6 \delta_{0} A}}{\e}, \frac{B}{\e} }.
    \qedhere
  \]
\end{proof}

\section{Convergence results for IG}
\label{sec:ig-convergence}
In this section we present results that are extremely similar to the previously obtained bounds for RR and SO. For completeness, we also provide a full description of IG in Algorithm~\ref{alg:ig}.
\begin{algorithm}[H]
    \caption{Incremental Gradient (IG)}
    \label{alg:ig}
 \begin{algorithmic}[1]
   \Require Stepsize $\gamma > 0$, initial vector $x_0 = x_0^0 \in \R^d$, number of epochs $T$
    \For{epochs $t=0,1,\dotsc,T-1$}
       \For{$i=0, 1, \ldots, n-1$}
          \State $x_{t}^{i+1} = x_t^{i} - \gamma \nabla f_{i+1} (x_t^i)$
       \EndFor
       \State $x_{t+1} = x_{t}^{n}$
    \EndFor
 \end{algorithmic}
 \end{algorithm}

\begin{theorem}
  \label{thm:ig}
  Suppose that Assumption~\ref{asm:f-smoothness} is satisfied. Then we have the following results for the Incremental Gradient method:
    
\begin{itemize}[leftmargin=0.15in,itemsep=0.01in,topsep=0pt]
  \item {\emone If each $f_i$ is $\mu$-strongly convex}: if $\gamma \leq \frac{1}{L}$, then
  \[ \sqn{x_{T} - x_\ast} \leq \br{1 - \gamma \mu}^{nT} \sqn{x_0 - x_\ast} + \frac{\gamma^2 L n^2  \sigmaesc^2}{\mu}. \]
  By carefully choosing the stepsize as in \Cref{corr:all-sc-rr-conv}, we see that this result implies that IG has sample complexity  $\ctO\br{\kappa + \frac{\sqrt{\kappa} n \sigmaesc}{\mu \sqrt{\e}}}$ in order to reach a point $\tilde{x}$ with $\norm{\tilde{x} - x_\ast}^2 \leq \e$.
  \item {\emone If $f$ is $\mu$-strongly convex and each $f_i$ is convex}: if $\gamma \leq \frac{1}{\sqrt{2} n L}$, then
  \[ \sqn{x_T - x_\ast} \leq \br{1 - \frac{\gamma \mu n}{2}}^{T} \sqn{x_0 - x_\ast} + 2 \gamma^2 \kappa  n^2 \sigmaesc^2.  \]
  Using the same approach for choosing the stepsize as \Cref{corr:all-sc-rr-conv}, we see that IG in this setting reaches an $\e$-accurate solution after $\ctO\br{ n \kappa + \frac{\sqrt{\kappa} n \sigmaesc}{\mu \sqrt{\e}} }$ individual gradient accesses.
  \item {\emone If each $f_i$ is convex}: if $\gamma \leq \frac{1}{\sqrt{2} n L}$, then
  \[ f(\hat{x}_T) - f(x_\ast) \leq \frac{\sqn{x_0 - x_\ast}}{2 \gamma n T} + \frac{\gamma^2 L n^2 \sigmaesc^2}{2} , \]
  where $\hat{x}_T \eqdef \frac{1}{T} \sum_{t=1}^{T} x_t$. Choosing the stepsize $\gamma = \min\pbr{\frac{1}{\sqrt{2} n L}, \frac{\sqrt{\e}}{\sqrt{L} n \sigmaesc} }$, then the average of iterate generated by IG is an $\e$-accurate solution (i.e., $f(\hat{x}_T) - f(x_\ast) \leq \e$) provided that the total number of iterations satisfies
  \[ n T \geq \frac{\sqn{x_0 - x_\ast}}{\e} \max \pbr{ \sqrt{8} n L, \frac{\sqrt{L} \sigmaesc n}{\sqrt{\e}} }. \]
  \item {\emone If each $f_i$ is possibly non-convex}: if Assumption~\ref{asm:2nd-moment} holds with constants $A, B \geq 0$ and $\gamma \leq \min \pbr{ \frac{1}{ \sqrt{8} n L}, \frac{1}{(4 L^2 n^3 A T)^{1/3}}}$, then
  \[ 
    \min_{t=0, \ldots, T-1} \sqn{\nabla f(x_t)} \leq \frac{12 \br{f(x_0) - f_\ast}}{\gamma n T} + 8 \gamma^2 L^2 n^2 B^2. \]
  Using an approach similar to \Cref{corr:rr-nonconvex}, we can establish that IG reaches a point with gradient norm less than $\e$ provided that the total number of iterations exceeds
  \[ n T \geq \frac{48 \br{f(x_0) - f_\ast} L n}{\e^2} \max \pbr{ \sqrt{2}, \frac{\sqrt{24 \br{f(x_0) - f_\ast}A}}{\e}, \frac{2 B}{\e}  }. \]
\end{itemize}
\end{theorem}

The proof of Thoerem~\ref{thm:ig} is given in the rest of the section, but first we briefly discuss the convergence rates and the relation of the result on strongly convex objectives to the lower bound of \citet{Safran2020good}.

\paragraph{Discussion of the convergence rates.} A brief comparison between the sample complexities given for IG in Theorem~\ref{thm:ig} and those given for RR (in Table~\ref{tab:conv-rates}) reveals that IG has similar rates to RR but with a worse dependence on $n$ in the variance term (the term associated with $\sigmaesc$ in the convex case and $B$ in the non-convex case), in particular IG is worse by a factor of $\sqrt{n}$. This difference is significant in the large-scale machine learning regime, where the number of data points $n$ can be on the order of thousands to millions.

\paragraph{Discussion of existing lower bounds.} \citet{Safran2020good} give the lower bound (in a problem with $\kappa = 1$) 
\[ \sqn{x_{T} - x_\ast} = \Omega\br{ \frac{\sigmaesc^2}{\mu^2 T^2} }. \]

This implies a sample complexity of $\cO\br{ \frac{n \sigmaesc}{\mu \sqrt{\e}}}$, which matches our upper bound (up to an extra iteration and log factors) in the case each $f_i$ is strongly convex and $\kappa = 1$.

\subsection{Preliminary Lemmas for Theorem~\ref{thm:ig}}
\subsubsection{Two lemmas for convex objectives}
\begin{lemma}
  \label{lem:convex_ig_deviation}
	Consider the iterates of Incremental Gradient (Algorithm~\ref{alg:ig}). Suppose that functions $f_1, \ldots, f_n$ are convex and that Assumption~\ref{asm:f-smoothness} is satisfied. Then it holds
  \begin{equation}
    \label{eq:lma-ig-cVt-bound}
    \sum_{k=0}^{n-1} \sqn{x_t^k - x_{t+1}} \leq 4 \gamma^2 L n^2 \sum_{i=0}^{n-1} D_{f_{i+1}} (x_\ast, x_t^i) + 2 \gamma^2 n^3 \sigmaesc^2,
  \end{equation}
  where $\sigmaesc^2$ is the variance at the optimum given by $\sigmaesc^2 \eqdef \frac{1}{n} \sum_{i=1}^{n} \sqn{\nabla f_{i} (x_\ast)}$.
\end{lemma}
\begin{proof}
	The proof of this Lemma is similar to that of Lemma~\ref{lem:convex_rr_deviation} but with a worse dependence on the variance term, since there is no randomness in IG. Fix any $k\in\{0,\dotsc, n-1\}$. It holds by definition
  \[
      x_t^k - x_{t+1}
      =\gamma\sum_{i=k}^{n-1}\nabla f_{i+1}(x_t^i)
      = \gamma\sum_{i=k}^{n-1}(\nabla f_{i+1}(x_t^i) - \nabla f_{i+1}(x_*))+\gamma\sum_{i=k}^{n-1}\nabla f_{i+1}(x_*).
  \]
  Applying Young's inequality to the sums above yields
  \begin{eqnarray*}
      \|x_t^k - x_{t+1}\|^2
      &\overset{\eqref{eq:sqnorm-triangle-2}}{\le}& 2\gamma^2\sqn{\sum_{i=k}^{n-1}(\nabla f_{i+1}(x_t^i) - \nabla f_{i+1}(x_*))} + 2\gamma^2\sqn{\sum_{i=k}^{n-1}\nabla f_{i+1}(x_*)} \\
      &\overset{\eqref{eq:sqnorm-sum-bound}}{\le}& 2\gamma^2n\sum_{i=k}^{n-1}\sqn{\nabla f_{i+1}(x_t^i) - \nabla f_{i+1}(x_*)} + 2\gamma^2\sqn{\sum_{i=k}^{n-1}\nabla f_{i+1}(x_*)}\\
      &\overset{\eqref{eq:bregman_lower_smooth}}{\le}& 4\gamma^2Ln\sum_{i=k}^{n-1}D_{f_{i+1}}(x_*, x_t^i) + 2\gamma^2\sqn{\sum_{i=k}^{n-1}\nabla f_{i+1}(x_*)} \\
      &\le& 4\gamma^2Ln\sum_{i=0}^{n-1}D_{f_{i+1}}(x_*, x_t^i) + 2\gamma^2\sqn{\sum_{i=k}^{n-1}\nabla f_{i+1}(x_*)} .
  \end{eqnarray*}
  Summing up,
  \begin{equation}
    \label{eq:ig-cVt-bound-proof-1}
    \sum_{k=0}^{n-1} \sqn{x_t^k - x_{t+1}} \leq 4\gamma^2Ln^2\sum_{i=0}^{n-1} D_{f_{i+1}} (x_\ast, x_t^i) + 2\gamma^2 \sum_{k=0}^{n-1} \sqn{\sum_{i=k}^{n-1}\nabla f_{i+1}(x_*)} .
  \end{equation}
  We now bound the second term in \eqref{eq:ig-cVt-bound-proof-1}. We have
  \begin{eqnarray}
    \sum_{k=0}^{n-1} \sqn{\sum_{i=k}^{n-1} \nabla f_{i+1} (x_\ast)} &\overset{\eqref{eq:sqnorm-sum-bound}}{\leq}& \sum_{k=0}^{n-1} (n-k) \sum_{i=k}^{n-1} \sqn{\nabla f_{i+1} (x_\ast)} \nonumber \\
    &\leq& \sum_{k=0}^{n-1} \br{n-k} \sum_{i=0}^{n-1} \sqn{\nabla f_{i+1} (x_\ast)} \nonumber \\
    &=& \sum_{k=0}^{n-1} \br{n-k} n \sigmaesc^2 = \frac{n^2 (n+1)}{2} \sigmaesc^2 \leq n^3 \sigmaesc^2.
    \label{eq:ig-cVt-bound-proof-2}
  \end{eqnarray} 
  Using \eqref{eq:ig-cVt-bound-proof-2} in \eqref{eq:ig-cVt-bound-proof-1}, we derive
  \[
    \sum_{k=0}^{n-1} \sqn{x_t^k - x_{t+1}} \leq 4 \gamma^2 L n^2 \sum_{i=0}^{n-1} D_{f_{i+1}} (x_\ast, x_t^i) + 2 \gamma^2 n^3 \sigmaesc^2.
    \qedhere
  \]
\end{proof}

\begin{lemma}
  \label{lem:convex-ig-recursion}
  Assume the functions $f_1, \ldots, f_n$ are convex and that Assumption~\ref{asm:f-smoothness} is satisfied. If Algorithm~\ref{alg:ig} is run with a stepsize $\gamma \leq \frac{1}{\sqrt{2} L n}$, then
  \[ \sqn{x_{t+1} - x_\ast} \leq \sqn{x_t - x_\ast} - 2 \gamma n \br{f(x_{t+1}) - f(x_\ast)} +  \gamma^3 L n^3 \sigmaesc^2. \]
\end{lemma}
\begin{proof}
  The proof for this lemma is identical to \Cref{lem:improved_convex} but with the estimate of \Cref{lem:convex_ig_deviation} used for $\sum_{i=0}^{n-1} \sqn{x_t^i - x_{t+1}}$ instead of \Cref{lem:convex_rr_deviation}. We only include it for completeness. Define the sum of gradients used in the $t$-th epoch as $g_{t} \eqdef \sum_{i=0}^{n-1} \nabla f_{i+1} (x_t^i)$. By definition of $x_{t+1}$, we have $x_{t+1} = x_t - \gamma g_t$. Using this,
	\begin{align*}
		\|x_{t}-x_*\|^2
		=\|x_{t+1}+\gamma g_t-x_*\|^2
		&=\|x_{t+1}-x_*\|^2+2\gamma\ev{g_t, x_{t+1}-x_*}+\gamma^2\|g_t\|^2 \\
		&\ge \|x_{t+1}-x_*\|^2+2\gamma\ev{g_t, x_{t+1}-x_*} \\
		&= \|x_{t+1}-x_*\|^2+2\gamma\sum_{i=0}^{n-1}\ev{\nabla f_{i+1}(x_t^i), x_{t+1}-x_*}.
	\end{align*}
	For any $i$ we have the following decomposition
	\begin{align}
		\ev{\nabla f_{i+1}(x_t^i), x_{t+1}-x_*}
    &= [f_{i+1}(x_{t+1})-f_{i+1}(x_*)]\\
    & \quad + [f_{i+1}(x_*) - f_{i+1}(x_t^i) - \<\nabla f_{i+1}(x_t^i), x_{t}^i-x_*>] \notag\\
		& \quad - [f_{i+1}(x_{t+1}) - f_{i+1}(x_t^i)-\<\nabla f_{i+1}(x_t^i), x_{t+1}-x_t^i>] \notag\\
		&= [f_{i+1}(x_{t+1})-f_{i+1}(x_*)] + D_{f_{i+1}}(x_*, x_t^i) - D_{f_{i+1}}(x_{t+1}, x_t^i). \label{eq:ig_decomposition}
	\end{align}
	Summing the first quantity in~\eqref{eq:ig_decomposition} over $i$ from $0$ to $n-1$ gives
	\begin{align*}
		\sum_{i=0}^{n-1}[f_{i+1}(x_{t+1})-f_{i+1}(x_*)]
		= n(f(x_{t+1})-f_*).
	\end{align*}
	Now let us work out the third term in the decomposition~\eqref{eq:ig_decomposition} using $L$-smoothness,
	\begin{align*}
		D_{f_{i+1}}(x_{t+1}, x_t^i)
		\le \frac{L}{2}\|x_{t+1}-x_t^i\|^2.
	\end{align*}
	We next sum the right-hand side over $i$ from $0$ to $n-1$ and use \Cref{lem:convex_ig_deviation}
	\begin{eqnarray*}
		\sum_{i=0}^{n-1} D_{f_{i+1}}(x_{t+1}, x_t^i)
    &\leq& \frac{L}{2} \sum_{i=0}^{n-1} \sqn{x_{t+1} - x_t^i} \\ 
		&\overset{\eqref{eq:lma-ig-cVt-bound}}{\le}&  2\gamma^2 L^2n^2\sum_{i=0}^{n-1} D_{f_{i+1}}(x_*, x_t^i)
		+\gamma^2  L n^3\sigma_*^2.
	\end{eqnarray*}
	Therefore, we can lower-bound the sum of the second and the third term in~\eqref{eq:ig_decomposition} as
	\begin{align*}
		\sum_{i=0}^{n-1}(D_{f_{i+1}}(x_*, x_t^i) - D_{f_{i+1}}(x_{t+1}, x_t^i))
    &\ge \sum_{i=0}^{n-1} D_{f_{i+1}}(x_*, x_t^i)  \\
    &\qquad - \br{2\gamma^2L^2n^2\sum_{i=0}^{n-1} D_{f_{i+1}}(x_*, x_t^i) - \gamma^2 L n^3\sigma_*^2} \\
		&= (1-2\gamma^2L^2n^2)\sum_{i=0}^{n-1} D_{f_{i+1}}(x_*, x_t^i) - \gamma^2 L n^3\sigma_*^2 \\
		&\ge  - \gamma^2 L n^3\sigma_*^2,
	\end{align*}
	where in the third inequality we used that $\gamma \leq \frac{1}{\sqrt{2} L n}$ and that $D_{f_{i+1}} (x_\ast, x_t^i)$ is nonnegative. Plugging this back into the lower-bound on $\|x_t-x_*\|^2$ yields
	\[
		\|x_t-x_*\|^2
		\ge \sqn{x_{t+1}-x_*} + 2\gamma n\br{f(x_{t+1})-f_*}- \gamma^3 L n^3\sigma_*^2.
	\]
	Rearranging the terms gives the result.
\end{proof}

\subsubsection{A lemma for non-convex objectives}

\begin{lemma}
  \label{lemma:vt-bound-ig}
  Suppose that Assumption~\ref{asm:f-smoothness} holds. Suppose that Algorithm~\ref{alg:ig} is used with a stepsize $\gamma > 0$ such that $\gamma \leq \frac{1}{2 L n}$. Then we have,
  \begin{equation}
    \label{eq:vt-bound-ig}
    \sum_{i=1}^{n} \sqn{x_t^i - x_t} \leq 4 \gamma^2 n^3 \sqn{\nabla f(x_t)} + 4 \gamma^2 n^3 \sigma_{t}^2,
  \end{equation}
  where $\sigma_{t}^2 \eqdef \frac{1}{n} \sum_{j=1}^{n} \sqn{\nabla f_{j} (x_t) - \nabla f(x_t)}$.
\end{lemma}
\begin{proof}
  Let $i \in \{ 1, 2, \ldots, n \}$. Then we can bound the deviation of a single iterate as,
  \begin{eqnarray*}
      \sqn{x_t^i - x_t} = \sqn{x_t^0 - \gamma \sum_{j=0}^{i-1} \nabla f_{j+1} (x_t^{j}) - x_t} &=& \gamma^2 \sqn{\sum_{j=0}^{i-1} \nabla f_{j+1} (x_t^{j})} \\
      &\overset{\eqref{eq:sqnorm-sum-bound}}{\leq}& \gamma^2 i \sum_{j=0}^{i-1} \sqn{\nabla f_{j+1} (x_t^{j})}.
  \end{eqnarray*}
  Because $i \leq n$, we have
  \begin{equation}
      \sqn{x_t^i - x_t} \leq \gamma^2 i \sum_{j=0}^{i-1} \sqn{\nabla f_{i+1} (x_t^{j})} \leq \gamma^2 n \sum_{j=0}^{i-1} \sqn{\nabla f_{i+1} (x_t^{j})} \leq \gamma^2 n \sum_{j=0}^{n-1} \sqn{\nabla f_{i+1} (x_t^{j})}.
      \label{eq:vt-ig-1}
  \end{equation}
  Summing up allows us to estimate $V_t$:
  \begin{eqnarray}
      V_{t} &=& \sum_{i=1}^{n} \sqn{x_t^i - x_t} \nonumber \\
      &\overset{\eqref{eq:vt-ig-1}}{\leq} & \sum_{i=1}^{n} \br{\gamma^2 n \sum_{j=0}^{n-1} \sqn{\nabla f_{j+1} (x_t^j)}} \nonumber \\ & = &\gamma^2 n^2 \sum_{j=0}^{n-1} \sqn{\nabla f_{j+1} (x_t^j)} \nonumber \\
      &\overset{\eqref{eq:sqnorm-triangle-2}}{\leq} & 2 \gamma^2 n^2 \sum_{j=0}^{n-1} \br{ \sqn{\nabla f_{j+1} (x_t^j) - \nabla f_{i+1} (x_t) } + \sqn{\nabla f_{j+1} (x_t) } } \nonumber \\
      &=& 2 \gamma^2 n^2 \sum_{j=0}^{n-1} \sqn{\nabla f_{j+1} (x_t^j) - \nabla f_{j+1} (x_t) } + 2 \gamma^2 n^2 \sum_{j=0}^{n-1} \sqn{\nabla f_{j+1} (x_t)}.
      \label{eq:vt-ig-2}
  \end{eqnarray}
  For the first term in \eqref{eq:vt-ig-2} we can use the smoothness of individual losses and that $x_t^0 = x_t$:
  \begin{align}
      \sum_{j=0}^{n-1} \sqn{\nabla f_{j+1} (x_t^j) - \nabla f_{j+1} (x_t)} &\overset{\eqref{eq:nabla-Lip}}{\leq}  L^2 \sum_{j=0}^{n-1}  \sqn{x_t^j - x_t} = L^2 \sum_{j=1}^{n-1} \sqn{x_t^j - x_t} = L^2 V_t.
      \label{eq:vt-ig-3}
  \end{align}
  The second term in \eqref{eq:vt-ig-2} is a sum over all the individual gradient evaluated at the same point $x_t$. Hence, we can drop the permutation subscript and then use the variance decomposition:
  \begin{eqnarray}
    \sum_{j=0}^{n-1} \sqn{\nabla f_{i+1} (x_t)} &=& \sum_{j=1}^{n} \sqn{\nabla f_{j} (x_t)} \nonumber \\
    &\overset{\eqref{eq:variance-decomp}}{=}&  n \sqn{\nabla f(x_t)} + \sum_{j=1}^{n} \sqn{\nabla f_{j} (x_t) - \nabla f(x_t)} \nonumber \\
    &=& n \sqn{\nabla f(x_t)} + n \sigma_{t}^2.
    \label{eq:vt-ig-4}
  \end{eqnarray}
  We can then use~\eqref{eq:vt-ig-3} and~\eqref{eq:vt-ig-4} in~\eqref{eq:vt-ig-2},
  \[
   V_t \leq 2 \gamma^2 L^2 n^2 V_t + 2 \gamma^2 n^3 \sqn{\nabla f(x_t)} + 2 \gamma^2 n^3 \sigma_{t}^2.
  \]
Since  $V_t$ shows up in both sides of the equation, we can rearrange to obtain
  \[
    \br{1 - 2 \gamma^2 L^2 n^2 } V_t \leq 2 \gamma^2 n^3 \sqn{\nabla f(x_t)} + 2 \gamma^2 n^3 \sigma_{t}^2.
  \]
  If $\gamma \leq \frac{1}{2 L n}$, then $1 - 2 \gamma^2 L^2 n^2  \geq \frac{1}{2}$ and hence
  \[
    V_{t} \leq 4 \gamma^2 n^3 \sqn{\nabla f(x_t)} + 4 \gamma^2 n^3 \sigma_{t}^2.
    \qedhere
  \]
\end{proof}

\subsection{Proof of Theorem~\ref{thm:ig}}
\begin{proof}
  \begin{itemize}[leftmargin=0.15in,itemsep=0.01in,topsep=0pt]
    \item {\emone If each $f_i$ is $\mu$-strongly convex}: The proof follows that of \Cref{thm:all-sc-rr-conv}. Define $$x_\ast^i = x_\ast - \gamma \sum_{j=0}^{i-1} \nabla f_{j+1} (x_\ast).$$ First, we have
    \begin{align*}
      &\|x_t^{i+1}-x_*^{i+1}\|^2\\
      &=\|x_t^{i}-x_*^{i}\|^2-2\gamma\<\nabla f_{i+1}(x_t^i)-\nabla f_{i+1}(x_*), x_t^i - x_*^i>+\gamma^2\|\nabla f_{i+1}(x_t^i) - \nabla f_{i+1}(x_*)\|^2.
    \end{align*}
    Using the same three-point decomposition as \Cref{thm:all-sc-rr-conv} and strong convexity, we have
    \begin{align*}
      - \ev{\nabla f_{i+1} (x_t^i) - \nabla f_{i+1} (x_\ast), x_t^i - x_\ast^i} &= - D_{f_{i+1}} (x_\ast^i, x_t^i) - D_{f_{i+1}} (x_t^i, x_\ast) + D_{f_{i+1}} (x_\ast^i, x_\ast) \\
      &\leq - \frac{\mu}{2} \sqn{x_t^i - x_\ast^i} - D_{f_{i+1}} (x_t^i, x_\ast) + D_{f_{i+1}} (x_\ast^i, x_\ast).
    \end{align*}
    Using smoothness and convexity
    \[ \frac{1}{2L} \sqn{\nabla f_{i+1} (x_t^i) - \nabla f_{i+1} (x_\ast)} \leq D_{f_{i+1}} (x_t^i, x_\ast). \]
    Plugging in the last two inequalities into the recursion, we get
    \begin{align}
      \sqn{x_t^{i+1} - x_\ast^{i+1}} &\leq \br{1 - \gamma \mu} \sqn{x_t^i - x_\ast^i} - 2 \gamma \br{1 - \gamma L} D_{f_{i+1}} (x_t^i, x_\ast) + 2 \gamma D_{f_{i+1}} (x_\ast^i, x_\ast). \nonumber \\
      &\le \br{1 - \gamma \mu} \sqn{x_t^i - x_\ast^i} + 2 \gamma D_{f_{i+1}} (x_\ast^i, x_\ast).
      \label{eq:ig-proof-allsc-1}
    \end{align}
    For the last Bregman divergence, we have
    \begin{eqnarray*}
      D_{f_{i+1}} (x_\ast^i, x_\ast) & \overset{\eqref{eq:L-smoothness}}{\leq} & \frac{L}{2} \sqn{x_\ast^i - x_\ast} \nonumber \\ &=& \frac{\gamma^2 L }{2} \sqn{\sum_{j=0}^{i-1} \nabla f_{j+1} (x_\ast)} \\
      &\overset{\eqref{eq:sqnorm-sum-bound}}{\leq}& \frac{\gamma^2 L  i}{2} \sum_{j=0}^{i-1} \sqn{\nabla f_{j+1} (x_\ast)} \\
      &=& \frac{\gamma^2 L  i n}{2} \sigmaesc^2 \leq \frac{\gamma^2 L  n^2}{2} \sigmaesc^2.
    \end{eqnarray*}
    Plugging this into \eqref{eq:ig-proof-allsc-1}, we get
    \[ \sqn{x_t^{i+1} - x_\ast^{i+1}} \leq \br{1 - \gamma \mu} \sqn{x_t^i - x_\ast^i} +  \gamma^3 L n^2 \sigmaesc^2. \]
    We recurse and then use that $x_\ast^n = x_\ast$, $x_{t+1} = x_t^n$, and that $x_\ast^0 = x_\ast$, obtaining
    \begin{align*}
      \sqn{x_{t+1} - x_\ast} = \sqn{x_t^n - x_\ast^n} &\leq \br{1 - \gamma \mu}^{n} \sqn{x_t^0 - x_\ast^0} + \gamma^3 L  n^2 \sigmaesc^2 \sum_{j=0}^{n-1} \br{1 - \gamma \mu}^{j} \\
      &= \br{1 - \gamma \mu}^{n} \sqn{x_t - x_\ast} +  \gamma^3 L n^2 \sigmaesc^2 \sum_{j=0}^{n-1} \br{1 - \gamma \mu}^{j}.
    \end{align*}
    Recursing again,
    \begin{align*}
      \sqn{x_{T} - x_\ast} &\leq \br{1 - \gamma \mu}^{n T} \sqn{x_0 - x_\ast} + \gamma^3 L  n^2 \sigmaesc^2 \sum_{j=0}^{n-1} \br{1 - \gamma \mu}^{j} \sum_{t=0}^{T-1} \br{1 - \gamma \mu}^{nt} \\
      &=  \br{1 - \gamma \mu}^{n T} \sqn{x_0 - x_\ast} +  \gamma^3 L n^2 \sigmaesc^2 \sum_{k=0}^{nT-1} \br{1 - \gamma \mu}^{k} \\
      &\leq \br{1 - \gamma \mu}^{n T} \sqn{x_0 - x_\ast} + \frac{\gamma^3 L  n^2 \sigmaesc^2}{\gamma \mu} \\
      &= \br{1 - \gamma \mu}^{n T} \sqn{x_0 - x_\ast} +\gamma^2  \kappa  n^2 \sigmaesc^2.
    \end{align*}
    \item {\emone If $f$ is $\mu$-strongly convex and each $f_i$ is convex}: the proof is identical to that of Theorem~\ref{thm:only-f-sc-rr-conv} but using Lemma~\ref{lem:convex-ig-recursion} instead of Lemma~\ref{lem:improved_convex}, and we omit it for brevity.
    \item {\emone If each $f_i$ is convex}: the proof is identical to that of Theorem~\ref{thm:weakly-convex-f-rr-conv} but using Lemma~\ref{lem:convex-ig-recursion} instead of Lemma~\ref{lem:improved_convex}, and we omit it for brevity.
    \item {\emone If each $f_i$ is possibly non-convex}: note that Lemma~\ref{lemma:epoch-recursion-non-convex} also applies to IG without change, hence if $\gamma \leq \frac{1}{Ln}$ we have
    \[ f(x_{t+1}) \leq f(x_t) - \frac{\gamma n}{2} \sqn{\nabla f(x_t)} + \frac{\gamma L^2}{2} \sum_{i=1}^{n} \sqn{x_t - x_t^i}. \]
    We may then apply \Cref{lemma:vt-bound-ig} to get for $\gamma \leq \frac{1}{2 L n}$
    \begin{align*}
      f(x_{t+1}) &\leq f(x_t) - \frac{\gamma n}{2} \sqn{\nabla f(x_t)} + \frac{\gamma L^2}{2} \br{4 \gamma^2 n^3 \sqn{\nabla f(x_t)} + 4 \gamma^2 n^3 \sigma_{t}^2} \\
      &= f(x_t) - \frac{\gamma n}{2} \br{1 - 4 \gamma^2 L^2 n^2} \sqn{\nabla f(x_t)} + 2 \gamma^3 L^2 n^3 \sigma_{t}^2.
    \end{align*}
    Using that $\gamma \leq \frac{1}{\sqrt{8} L n}$ and subtracting $f_\ast$ from both sides, we derive
    \[
      f(x_{t+1}) - f_\ast \leq \br{f(x_t) - f_\ast} - \frac{\gamma n}{4} \sqn{\nabla f(x_t)} + 2  \gamma^3 L^2 n^3 \sigma_{t}^2.
    \]
    Using Assumption~\ref{asm:2nd-moment}, we get
    \begin{equation}
      f(x_{t+1}) - f_\ast \leq \br{1 + 4 \gamma^3 L^2 A n^3 } \br{f(x_t) - f_\ast} - \frac{\gamma n}{4} \sqn{\nabla f(x_t)} + 2 \gamma^3 L^2 n^3 B^2.
      \label{eq:ig-nonconvex-thm-1}
    \end{equation}
    Applying Lemma~\ref{lemma:noncvx-recursion-solution} to \eqref{eq:ig-nonconvex-thm-1}, thus, gives
    \begin{equation}
      \min_{t=0, \ldots, T-1} \sqn{\nabla f(x_t)} \leq \frac{4 \br{1 + 4 \gamma^3 L^2 A n^3 }^{T}}{\gamma n T} \br{f(x_0) - f_\ast} + 8 \gamma^2 L^2 n^2 B^2.
      \label{eq:ig-nonconvex-thm-2}
    \end{equation}
    Note that by our assumption on the stepsize, $4 \gamma^3 L^2 A n^3  T \leq 1$, hence,
    \[ \br{1 + 4 \gamma^3 L^2 A n^3 }^{T} \leq \exp\br{4 \gamma^3 L^2 A n^3  T} \leq \exp\br{1} \leq 3. \]
    It remains to use this in \eqref{eq:ig-nonconvex-thm-2}.
    \qedhere
  \end{itemize}
\end{proof}

\end{document}